\documentclass[10pt,reqno]{amsart}

\usepackage[utf8]{inputenc}
\usepackage[english]{babel}
\usepackage{amsmath,amsfonts,amssymb,amsthm,shuffle}
\usepackage[T1]{fontenc}
\usepackage{lmodern}
\usepackage{mathtools}
\usepackage[titletoc]{appendix}

\usepackage[top=3.5cm,bottom=3.5cm,left=3.6cm,right=3.6cm]{geometry}

\usepackage[dvipsnames]{xcolor}
\usepackage[hyperindex=true,frenchlinks=true,colorlinks=true,
citecolor=Mahogany,linkcolor=DarkOrchid,urlcolor=Tan,linktocpage,
pagebackref=true]{hyperref}

\usepackage{tikz}
\usetikzlibrary{shapes}
\usetikzlibrary{fit}
\usetikzlibrary{decorations.pathmorphing}

\usepackage{dsfont}
\usepackage{wasysym}
\usepackage{stmaryrd}
\usepackage{cite}
\usepackage{subfigure}
\usepackage{multirow}
\usepackage{enumitem}
\usepackage{multicol}
\usepackage{bm}

\title[]{Bound  on the maximal function associated to the law of the iterated logarithms for Bernoulli random fields}
\keywords{random fields, bounded law of the iterated logarithms}
\subjclass[2010]{60G60; 60G10}
\date{\today}
\author{Davide Giraudo}

\numberwithin{equation}{subsection}
\renewcommand{\leq}{\leqslant}
\renewcommand{\geq}{\geqslant}

\newtheorem{Theorem}{Theorem}[section]
\newtheorem{Th\'eor\`eme}{Th\'eor\`eme}[section]
\newtheorem{Proposition}[Theorem]{Proposition}
\newtheorem{Lemma}[Theorem]{Lemma}

\newtheorem{Definition}[Theorem]{Definition}
\newtheorem{D\'efinition}[Th\'eor\`eme]{D\'efinition}
\newtheorem{Corollary}[Theorem]{Corollary}

\theoremstyle{remark}

\tikzstyle{Vertex}=[circle,draw=LimeGreen!80,fill=LimeGreen!8,
inner sep=1pt,minimum size=2mm,line width=1pt,font=\scriptsize]
\tikzstyle{Node}=[Vertex,draw=RoyalBlue!80,fill=RoyalBlue!8,inner sep=1.5pt]
\tikzstyle{Leaf}=[rectangle,draw=Black!70,fill=Black!16,
inner sep=0pt,minimum size=1mm,line width=1.25pt]
\tikzstyle{Edge}=[Maroon!80,cap=round,line width=1pt]
\tikzstyle{Mark1}=[draw=BrickRed!80,fill=BrickRed!8]
\tikzstyle{Mark2}=[draw=BurntOrange!80,fill=BurntOrange!8]
\tikzstyle{EdgeRew}=[->,RedOrange!80,cap=round,thick]

\newcommand{\Fca}{\mathcal{F}}
\newcommand{\Gca}{\mathcal{G}}
\newcommand{\Hca}{\mathcal{H}}

\newcommand \ens[1]{\left\{ #1\right\}}
\newcommand \R{\mathbb R}

\newcommand \N{\mathbb N}
\newcommand \PP{\mathbb P}
\newcommand{\el}{\mathbb L}
\newcommand{\E}[1]{\mathbb E\left[#1\right]}

\newcommand \Z{\mathbb Z}

\newcommand \abs[1]{\left|#1\right|}
\newcommand \eps{\varepsilon}

\newcommand{\f}{\mathcal F}

\newcommand{\pr}[1]{\left(#1\right)}
\newcommand{\norm}[1]{\left\lVert #1 \right\rVert}

\newcommand{\gr}[1]{\bm{#1}}
\newcommand{\gri}{\bm{i}}
\newcommand{\grj}{\bm{j}}
\newcommand{\grn}{\bm{n}}

\newcommand{\gru}{\bm{u}}
\newcommand{\gra}{\bm{a}}
\newcommand{\grb}{\bm{b}}

\newcommand{\imd}{\preccurlyeq}
\newcommand{\smd}{\succcurlyeq }

\begin{document}
\begin{abstract}
We provide a sufficient condition for the bounded law of the iterated logarithms for
strictly stationary random fields expressable as a functional of i.i.d.
random fields when the summation is done on rectangles. The study is done via the control of the moments of an appropriated maximal function.  Applications to functionals of linear random fields, functions of a Gaussian linear 
random field and Volterra process are given.  
\end{abstract} 
\maketitle

\section{Goal of the paper and main results}

 \subsection{Bounded law of the iterated logarithms for random fields} 

Let $d\geq 1$ be an integer and let $\pr{X_{\gri}}_{\gri\in\Z^d}$ be a random 
field and denote the partial sums 
\begin{equation}\label{eq:defn_sommes_part}
 S_{\grn}:=\sum_{\gr{1}\imd\gri \imd\grn}X_{\gri}, \quad  \grn\smd \gr{1} , 
\end{equation}
where $\imd$ denotes the coordinatewise order on 
the elements of $\Z^d$, that is,  for $\gri=\pr{i_q}_{q=1}^d$ 
and $\grj=\pr{j_q}_{q=1}^d$, 
$\gri\imd\grj$ if $i_q\leq j_q$ for all $q\in [d]:=\ens{1,\dots,d}$ (and similarly, 
we write $\gri\smd\grj$ if $i_q\geq j_q$ for all $q\in [d]$) and $\gr{1}=\pr{1,\dots,1}$.
The understanding of the behavior of such partial sums has received attention 
in the past years. When the partial sums \eqref{eq:defn_sommes_part} are 
normalized by $\abs{\grn}=\prod_{i=1}^dn_i$ and the random field $\pr{X_{\gri}}_{\gri \in \Z^d}$ is strictly stationary, functional central
 limit theorems have been established under various dependence structures: martingale differences with respect to the 
 lexicographic order  \cite{MR1875665,MR3473103}, martingale differences \cite{MR542479,MR1629903}, orthomartingale 
 differences and via orthomartingale approximation (see \cite{MR3264437,MR3504508,MR3427925,MR3769664,MR3798239}). 

Concerning the law of the iterated logarithms, it has been shown in \cite{MR0394894}
that for an i.i.d. collection of centered random variables 
$\pr{X_{\gri}}_{\gri\in\Z^d}$, (with $d>1$)  the following equivalence holds:
\begin{multline}\label{eq:LLI_gen_iid}
\E{X_{\gr{0}}^2\pr{L\pr{\abs{X_{\gr{0}}}     }     }^{d-1}/LL
\pr{\abs{X_{\gr{0}}}   }}<+\infty  \\
\Leftrightarrow
\limsup_{\grn \to +\infty}\frac{1}{\sqrt{\abs{\grn}L L\pr{\abs{\grn}    }  }}
S_{\grn}=\norm{X_{\gr{0}}}_2\sqrt d=
-\liminf_{\grn \to +\infty}\frac{1}{\sqrt{\abs{\grn}L L\pr{\abs{\grn}    }  }}
S_{\grn},
\end{multline}
where $L\colon \pr{0,+\infty}\to \R$ is defined by 
$L\pr{x}=\max\ens{\ln x,1}$ and $LL\colon \pr{0,+\infty}\to \R$ 
by $LL\pr{x}=L\circ L\pr{x}$,
and 
for a family of numbers $\pr{a_{\grn}}_{\grn\smd \gr{1}}$,  
$\limsup_{\grn\to +\infty}a_{\grn}:=\lim_{m\to +\infty}\sup_{\grn\smd m\gr{1}}a_{\grn}$ 
and similarly for $\liminf$.

In particular, the moment condition as well as the $\limsup$/$\liminf$ depend on the 
dimension $d$ and the normalization by $ \sqrt{\abs{\grn}  LL\pr{\abs{\grn}}  }$ 
is the best possible among those guaranting the finiteness of the $\limsup$/$\liminf$
in \eqref{eq:LLI_gen_iid}. Some related results have been obtained under other 
dependence structures; see \cite{MR3438530,MR2261554,MR1788546,MR1674964} for instance. 

In general, finding $\limsup_{\grn \to +\infty}\frac{1}{\sqrt{\abs{\grn}L L\pr{\abs{\grn}    }  }}
S_{\grn}$ is a hard task. Another one consists in establishing finiteness of 
$\sup_{\grn\in\N^d}\frac{1}{\sqrt{\abs{\grn}L L\pr{\abs{\grn}    }  }}
\abs{S_{\grn}}$ (where $\N$ denotes the set of positive integers) or integrability of this random variable.
More precisely, we would like to find a sufficient condition on the moments and 
the dependence of a stationary random field such that the quantity

\begin{equation}\label{eq:dnf_supremum_somme_normalisee}
\norm{\sup_{\grn\in\N^d}\frac 1{ \sqrt{\abs{\grn}  LL\pr{\abs{\grn}}  }}
\abs{S_{\grn}}}_p<+\infty, 1< p<2,
\end{equation}
is finite. When $d=1$ and $\pr{X_i}_{i\geq 1}$ is i.i.d. centered, it has been shown in 
\cite{MR0501237} that for $1<p<2$,
\begin{equation}
\norm{\sup_{n\geq 1}\frac 1{ \sqrt{n  LL\pr{n}  }}\abs{\sum_{i=1}^nX_i}}_p\leq c_p\norm{X_1}_2.
\end{equation}

This has been extended to martingales in \cite{MR3322323}, and to higher moments in 
\cite{MR2277288}.

In this paper, we will be concentrated in the following questions. First, 
we would like to give bound on the quantity
involved in \eqref{eq:dnf_supremum_somme_normalisee} in the i.i.d. case. Results 
in the one dimensional case are known, but to the best of
 our knowledge, it seems that no results are available in dimension 
 greater than one. Once this is done for i.i.d. random fields, a similar 
 question can be treated for a strictly stationary random field
 which can be expressed as a functional of finitely many i.i.d. 
 random variables, and then extend this to more general random fields, 
 which are functionals of an i.i.d. collection of random variable indexed by 
 $\Z^d$.

\subsection{Main results}

 We assume that 
$X_{\gr{i}}$ has the form $f\pr{\pr{\eps_{\gr{i}-\gr{j}}}_{\gr{j}\in\Z^d}}$ where 
$f\colon \R^{\Z^d}\to \R$ is measurable (with $\R^{\Z^d}$ endowed with the 
product topology) and $\pr{\eps_{\gr{u}}}_{\gr{u}\in\Z^d}$ is an
 independent identically distributed random field. This class 
 of random fields has been studied during the previous years, including a lot of 
 work in the area of functional central limit theorems \cite{MR2988107,MR3256190, 
 MR3483738}. When the summation on rectangles is considered, moments of order two 
 are sufficient and moment of order $p>2$ can be required for other types of partial 
 sum process.

The condition for the control of the maximal function associated to the law of the iterated logarithms 
involved in \eqref{eq:dnf_supremum_somme_normalisee}
 will require slightly more than finite moments of order $2$. In order to state it,
we define  for $p>1$ and $r\geq 0$, the function $\varphi_{p,r}\colon 
[0,+\infty)$ by $\varphi_{p,r}\pr{x}:=x^p\pr{1+\log\pr{1+x}    }^r$ and 
 denote by 
$\el_{p,r}$ the Orlicz space associated to this function. We define the 
norm $\norm{\cdot}_{p,r}$ of an element $X$ of $\el_{p,r}$ by 
\begin{equation}
 \norm{X}_{p,r}:=\inf\ens{\lambda>0\mid 
 \E{\varphi_{p,r}\pr{\frac X\lambda}  }\leq 1}.
\end{equation}
 
Denote also for $\gri=\pr{i_q}_{q=1}^d$ the quantity $\norm{\gri}_\infty:=\max_{1\leq q\leq d}
\abs{i_q}$ and  $\gr{0}:=\pr{0,\dots,0}$.

\begin{Theorem}\label{thm:cas_fonctionnelle_iid}
Let $\pr{X_{\gr{i}}}_{\gr{i}\in \Z^d}$ be a centered random field such that there exist 
an i.i.d. collection of random variables $\ens{\eps_{\gr{u}},\gr{u}\in\Z^d}$ and 
a measurable function $f\colon \R^{\Z^d}\to \R$ such that $X_{\gr{i}}= 
f\pr{\pr{\eps_{\gr{i}-\gr{j}}}_{\gr{j}\in\Z^d}}$. For all $1< p<2$, the following inequality holds:
\begin{equation}
\norm{ 
\sup_{\grn\in \N^d}\frac 1{ \sqrt{\abs{\grn} L L\pr{\abs{\grn}}   }    }
\abs{\sum_{\gr{1}\imd\gri \imd\grn  }X_{\gr{i}}   }
}_p\leq c_{p,d}\sum_{j\geq 0} \pr{j+1}^{d/2}\norm{X_{\gr{0},j  }}_{2,d-1},
\end{equation}
where $c_{p,d}$ depends only on $p$ and $d$ and 
\begin{equation}\label{eq:definitionde_X0j}
X_{\gr{0},j}=\E{X_{\gr{0}}\mid \sigma\ens{\eps_{\gr{u}},\norm{\gr{u}}_\infty\leq 
 j   }}-\E{X_{\gr{0}}\mid \sigma\ens{\eps_{\gr{u}},\norm{\gr{u}}_\infty\leq 
 j -1  }} , \quad j\geq 1;
\end{equation}
 \begin{equation}\label{eq:definitionde_X00}
 X_{\gr{0},0}:=\E{X_{\gr{0}}\mid \sigma\ens{\eps_{\gr{0}} }}.
 \end{equation}
\end{Theorem}
 
Observe that by the martingale convergence theorem, the sequence
 $\pr{\norm{X_{\gr{0},j  }}_{2,d-1}}_{j\geq 1}$ converges to $0$ provided that 
$X_{\gr{0}}\in \mathbb L_{2,d-1}$.

We will now formulate a result in the spirit of Theorem~\ref{thm:cas_fonctionnelle_iid} 
for subsets which can be expressed as a finite union of disjoint rectangles. We first present 
a result for i.i.d. random fields where the summation is done on subset of $\Z^d$ under 
an assumption on the cardinal.
A central limit theorem for linear random fields has been established in \cite{MR2832921} and 
a local central limit theorem in \cite{FPS20}. The case of short and long range dependent random field 
were adressed. In the short range dependent case, that is, $\sum_{\gri\in\Z^d}\abs{a_{\gri}}<\infty$, 
no condition on the size of the summation subset was made. However, one needs a bound on the size 
of the sets for the control of the maximal function, even in the case $d=1$. Indeed, 
taking $\Lambda_n$ pairwise disjoint and and i.i.d. $\pr{X_{\gri}}_{\gri\in\Z^d}$ where $X_{\gri}$ is standard 
normal, the random variable $\sup_{n\geq 1}\frac 1{ \sqrt{\ell_n LL\pr{\ell_n}  }    }
\abs{\sum_{\gri\in \Lambda_n }X_{\gr{i}}   }$ has the same distribution as 
$\sup_{n\geq 1}\frac 1{ \sqrt{  LL\pr{\ell_n}  }    }
\abs{ N_n    }$, where $\pr{N_n}_{n\geq 1}$ is an i.i.d. sequence of standard normal random variables. Therefore, 
by the second Borel-Cantelli lemma, the random variable 
$\sup_{n\geq 1}\frac 1{ \sqrt{\ell_n LL\pr{\ell_n}  }    }
\abs{\sum_{\gri\in \Lambda_n }X_{\gr{i}}   }$ is almost surely finite if and only if for some $M$, 
$\sum_{n\geq 1}\PP\ens{\abs{N_n}>M \sqrt{ LL\pr{\ell_n}  }}$ is finite. Using a lower tail inequality for the distribution function of a normal 
random variable, this imposes the convergence of the series $\sum_{n\geq 1}  
\exp\pr{-M^2L\pr{\ell_n}/2}/\sqrt{ LL\pr{\ell_n}  }$.

\begin{Theorem}\label{thm:linear_process_arbitrary_subsets}
 Let $\pr{\eps_{\gru}}_{\gru\in\Z^d}$ be an i.i.d. centered random field having a finite variance, let 
 $\pr{a_{\gri}}_{\gri\in\Z^d}$ be an absolutely summable sequence of integers and denote by $\pr{X_{\gri}}_{\gri\in\Z^d}$ the 
 linear process given by $X_{\gri}:=\sum_{\grj\in \Z^d}a_{\grj}\eps_{\gri-\grj}$. 
 Let $\pr{\Lambda_n}_{n\geq 1}$ be a 
sequence of finite subsets of $\Z^d$ and denote by $\ell_n$ the cardinal of $\Lambda_n$. Suppose that 
there exists $\delta>0$ such that $\ell_{n+1}\geq \ell_n\geq \exp\pr{n^\delta}$ and $C>0$ such that 
$\sum_{k=1}^n\sqrt{\ell_k/LL\pr{\ell_k}}\leq C\ell_n/LL\pr{\ell_n}$.
For each $p\in\pr{1,2}$, there exists 
 a constant $K_p$ depending  only on $p$ such that 
 \begin{equation}\label{eq:control_fct_max_somme_sur_ens_processus_lineaire}
\norm{ 
\sup_{n\geq 1}\frac 1{ \sqrt{\ell_n LL\pr{\ell_n}  }    }
\abs{\sum_{\gri\in \Lambda_n }X_{\gr{i}}   }
}_p\leq K\pr{p}C^{1/p}\delta^{-1/2}\norm{\eps_{\gr{0}}}_2\sum_{\grj\in\Z^d}\abs{a_{\grj}}.
\end{equation}
\end{Theorem}

Note that the condition on $\ell_n$ is satisfied for instance when there exists positive constants 
$c_1$ and $c_2$ such that for each $n\geq 1$, $c_12^n\leq \ell_n\leq c_22^n$, or more generally, 
if $c_1\left\lfloor a^n\right\rfloor\leq \ell_n\leq c_2\left\lfloor a^n\right\rfloor$ for some $a>1$, where 
$\left\lfloor x\right\rfloor$ is the unique integer for which $\left\lfloor x\right\rfloor\leq x<\left\lfloor x\right\rfloor+1$.

We now present a result for arbitrary functionals of a linear random field, with the counterpart 
that the summation sets are assumed to be finite disjoint unions of rectangles.

\begin{Theorem}\label{thm:cas_fonctionnelle_union_rectangles}
 Let $\pr{X_{\gr{i}}}_{\gr{i}\in \Z^d}$ be a centered random field such that there exist 
an i.i.d. collection of random variables $\ens{\eps_{\gr{u}},\gr{u}\in\Z^d}$ and 
a measurable function $f\colon \R^{\Z^d}\to \R$ such that $X_{\gr{i}}= 
f\pr{\pr{\eps_{\gr{i}-\gr{j}}}_{\gr{j}\in\Z^d}}$. Let $\pr{\Gamma_n}_{n\geq 1}$ be a 
sequence of subsets of $\Z^d$ such that for each $n\geq 1$, $\Gamma_n$ is the disjoint union 
of $\Gamma_{n}\pr{w}, 1\leq w\leq J_n$ and 
\begin{equation}
 \Gamma_{n}\pr{w}=\Z^d\cap \prod_{q=1}^d \left[\underline{n}_q\pr{w,n},\overline{n}_q\pr{w,n}\right], \overline{n}_q\pr{w,n},\underline{n}_q\pr{w,n}\in\Z.
\end{equation}
Assume that $\ell_n:=\operatorname{Card}\pr{\Gamma_n}$ satisfies $\ell_{n+1}\geq \ell_n\geq \exp\pr{n^\delta}$,
$\sum_{k=1}^n\sqrt{\ell_k/LL\pr{\ell_k}}\leq C\ell_n/LL\pr{\ell_n}$, where $C>0$ and $\delta>0$ are independent of $n$, and 
$\overline{n}_q\pr{w,n}-\underline{n}_q\pr{w,n}\geq 4$ for each $n\geq 1$, $1\leq w\leq J_n$, $1\leq q\leq d$.

Then for all $1<p<2$, the following inequality holds:
\begin{equation}\label{eq:control_fct_max_somme_sur_ens}
\norm{ 
\sup_{n\geq 1}\frac 1{ \sqrt{\ell_n LL\pr{\ell_n}  }    }
\abs{\sum_{\gri\in \Gamma_n }X_{\gr{i}}   }
}_p
\leq K\pr{p,d,C,\delta}\sum_{j\geq 0} \pr{j+1}^{d}L\pr{j}^{1/p}\norm{X_{\gr{0},j  }}_{2},
\end{equation}
where $X_{\gr{0},j}$ and $X_{\gr{0},0}$ are defined respectively by \eqref{eq:definitionde_X0j} 
and \eqref{eq:definitionde_X00} and $K\pr{p,d,C,\delta}$ depends on $p$, $d$, $C$ and $\delta$.
\end{Theorem}

Notice that unlike the case of the supremum on rectangles addressed in Theorem~\ref{thm:cas_fonctionnelle_iid}, 
the $\mathbb L^2$-norm of the random variables $X_{\gr{0},j  }$ is involded instead of the 
norm $\norm{\cdot}_{2,d-1}$. However, we need a condition on the size of the sets. Moreover, the weight term 
is $\pr{j+1}^{d}L\pr{j}^{1/p}$ which is stronger than the term $\pr{j+1}^{d/2}$ appearing in 
Theorem~\ref{thm:cas_fonctionnelle_iid}.

The terms $\norm{X_{\gr{0},j  }}_{2,d-1}$ and $\norm{X_{\gr{0},j  }}_{2}$  can be estimated by the so-called
physical measure of dependence, introduced in \cite{MR2172215}. 

\begin{Definition}
Let $\pr{X_{\gri}}_{\gri\in \Z^d}$ be a strictly stationary random field which can 
be expressed as a functional of an i.i.d. random field, that is, there exist an i.i.d. random field 
$\pr{\eps_{\gr{u}}}_{\gru\in \Z^d}$ and a measurable function $f\colon \R^{\Z^d}\to\R$ such that 
$X_{\gri}=f\pr{ \pr{\eps_{\gri -\gr{u}}}_{\gru\in \Z^d}   } $. Let $\eps'_{\gr{0}}$ be a random variable 
independent of $\pr{\eps_{\gr{u}}}_{\gru\in \Z^d}$. Denote by $\eps^*_{\gru}$ the random variable 
$\eps_{\gru}$ if $\gru\neq \gr{0}$ and $\eps'_{\gr{0}}$ if $\gru=\gr{0}$. 

For $r\geq 0$, we define the physical measure of dependence of $\pr{X_{\gri}}_{\gri\in \Z^d}$ by 
\begin{equation}\label{eq:definition_phy_dep_meas}
\delta_{2,r}\pr{\gri}:=\norm{f\pr{ \pr{\eps_{\gri -\gr{u}}}_{\gru\in \Z^d}   }-
f\pr{ \pr{\eps^*_{\gri -\gr{u}}}_{\gru\in \Z^d}   }        }_{2,r}.
\end{equation}

\end{Definition}

When the Orlicz norm $\norm{\cdot}_{2,r}$ is replaced by the $\el^p$-norm, there are various 
examples of random fields where the measure of dependence is estimated (see Section~2 in 
\cite{MR3256190} and \cite{MR2988107}). By using an appropriated version of Burkholder and Rosenthal's inequality 
in these spaces (see for instance Corollary~\ref{cor:Burlholder_Orlicz}), we can also 
estimate $\delta_r\pr{\gri}$. This approach also allows to bound $\norm{X_{\gr{0},j  }}_{2,d-1}$ 
by the coefficients $\delta_{d-1}\pr{\gri}$, like in the proof of Corollary~1 in \cite{MR3963881}.
This leads to the following result.

\begin{Corollary}\label{cor:coeff_phys_dep}
Let $\pr{X_{\gr{i}}}_{\gr{i}\in \Z^d}$ be a centered random field such that there exist 
an i.i.d. collection of random variables $\ens{\eps_{\gr{u}},\gr{u}\in\Z^d}$ and 
a measurable function $f\colon \R^{\Z^d}\to \R$ such that $X_{\gr{i}}= 
f\pr{\pr{\eps_{\gr{i}-\gr{j}}}_{\gr{j}\in\Z^d}}$. For all $1< p<2$, the following inequality holds:
\begin{equation}
\norm{ 
\sup_{\grn\in \N^d}\frac 1{ \sqrt{\abs{\grn} L L\pr{\abs{\grn}}   }    }
\abs{\sum_{\gr{1}\imd\gri \imd\grn  }X_{\gr{i}}   }
}_p\leq c_{p,d}\sum_{j\geq 0} \pr{j+1}^{d/2}\sqrt{ \sum_{\gri\in \Z^d, \norm{\gri}_\infty=j}
\delta_{2,d-1}\pr{\gri}^2
    } ,
\end{equation}
where $c_{p,d}$ depends only on $p$ and $d$ and $\delta_{2,d-1}\pr{\gri}$ is 
defined by \eqref{eq:definition_phy_dep_meas}.
\end{Corollary}

Let us also mention the following consequence of Theorem~\ref{thm:cas_fonctionnelle_union_rectangles}.

\begin{Corollary}\label{cor:coeff_phys_dep_union_of_rectangles}
Let $\pr{X_{\gr{i}}}_{\gr{i}\in \Z^d}$ be a centered random field such that there exist 
an i.i.d. collection of random variables $\ens{\eps_{\gr{u}},\gr{u}\in\Z^d}$ and 
a measurable function $f\colon \R^{\Z^d}\to \R$ such that $X_{\gr{i}}= 
f\pr{\pr{\eps_{\gr{i}-\gr{j}}}_{\gr{j}\in\Z^d}}$. Let $\pr{\Gamma_n}_{n\geq 1}$ be a 
sequence of subsets of $\Z^d$ such that for each $n\geq 1$, $\Gamma_n$ is the disjoint union 
of $\Gamma_{n}\pr{w}, 1\leq w\leq J_n$ and 
\begin{equation}
 \Gamma_{n}\pr{w}=\Z^d\cap \prod_{q=1}^d \left[\underline{n}_q\pr{w,n},\overline{n}_q\pr{w,n}\right].
\end{equation}
Assume that $\ell_n:=\operatorname{Card}\pr{\Gamma_n}$ satisfies $\ell_{n+1}\geq \ell_n\geq \exp\pr{n^\delta}$,
$\sum_{k=1}^n\sqrt{\ell_k/LL\pr{\ell_k}}\leq C\ell_n/LL\pr{\ell_n}$, where $C>0$ and $\delta>0$ are independent of $n$, and 
$\overline{n}_q\pr{w,n}-\underline{n}_q\pr{w,n}\geq 4$ for each $n\geq 1$, $1\leq w\leq J_n$, $1\leq q\leq d$.

Then for all $1<p<2$, the following inequality holds:
\begin{multline}\label{eq:control_fct_max_somme_sur_ens_phys_dep}
\norm{ 
\sup_{n\geq 1}\frac 1{ \sqrt{\ell_n LL\pr{\ell_n}  }    }
\abs{\sum_{\gri\in \Gamma_n }X_{\gr{i}}   }
}_p\\
\leq K\pr{p,d,C,\delta}\sum_{j\geq 0} \pr{j+1}^{d}L\pr{j}^{1/p}\sqrt{ \sum_{\gri\in \Z^d, \norm{\gri}_\infty=j}
\delta_{2,0}\pr{\gri}^2
    },
\end{multline}
where $\delta_{2,0}\pr{\gri}$ is 
defined by \eqref{eq:definition_phy_dep_meas}.
\end{Corollary}

\section{Applications}

In some particular cases of functionals of independent random fields, one can  
estimate $\norm{X_{\gr{0},j  }}_{2,d-1}$ in terms of the parameters of the 
considered model. We will focus on the case of linear and Volterra random fields and 
functions of a Gaussian linear random field.

\subsection{Functional of linear random fields}

We say that the random fields $\pr{X_{\gri}}_{\gri\in\Z^d}$ is a linear 
random field if there exists an i.i.d. centered random field $\pr{\eps_{\gri}}_{\gri
\in\Z^d}$ of square integrable random variables and a family of real numbers 
$\pr{a_{\gri}}_{\gri\in\Z^d}$ such that $\sum_{\gri\in\Z^d}a_{\gri}^2$ is finite and 
\begin{equation}\label{eq:dfn_linear_random_field}
X_{\grj}=\sum_{\gri\in\Z^d}a_{\gri}\eps_{\grj-\gri} \mbox{ a.s.}.
\end{equation}

We will give a sufficient condition for the control of the maximal function of 
a random field which can be expressed as a Lipschitz function of a 
linear random field.

\begin{Corollary}\label{cor:LLI_linear}
Let $\pr{X_{\gri}}_{\gri\in\Z^d}$ be a linear random field defined by 
\eqref{eq:dfn_linear_random_field}, where  $\pr{\eps_{\gri}}_{\gri
\in\Z^d}$ is i.i.d.. 
Let $g\colon \R\to \R$ be a $\gamma$-Hölder continuous function where $0<\gamma\leq 1$
and assume that $\eps_{\gr{0}}\in \mathbb L_{2/\gamma,d-1}$ and $\E{g\pr{X_{\gr{0}}}}=0$. Then for  
$1<p<2$, 
\begin{multline}
\norm{ 
\sup_{\grn\in \N^d}\frac 1{ \sqrt{\abs{\grn} L L\pr{\abs{\grn}}
}    }
\abs{\sum_{\gr{1}\imd\gri \imd\grn  }g\pr{X_{\gr{i}}   }}
}_p\\
\leq c_{p,d,\gamma}\sum_{j=0}^{+\infty} \pr{j+1}^{d/2} \pr{
\sum_{\gri\in\Z^d,\norm{\gri}_\infty=j}  \abs{a_{\gri }}^{2\gamma}}^{1/2}\norm{\eps_{\gr{0}}}_{2\gamma,d-1}  ,
\end{multline}
where $c_{p,d}$ depends only on $p$, $d$ and $\gamma$.
\end{Corollary}
 A similar result can be derived when the summation sets are finite disjoint unions of rectangles. 
 
 \begin{Corollary}\label{cor:LLI_linear_sum_on_disjoint_unions_of_rectangles}
Let $\pr{X_{\gri}}_{\gri\in\Z^d}$ be a linear random field defined by 
\eqref{eq:dfn_linear_random_field}, where  $\pr{\eps_{\gri}}_{\gri
\in\Z^d}$ is i.i.d.. 
Let $g\colon \R\to \R$ be a $\gamma$-Hölder continuous function where $0<\gamma\leq 1$
and assume that $\eps_{\gr{0}}\in \mathbb L_{2/\gamma,d-1}$ and $\E{g\pr{X_{\gr{0}}}}=0$. 
Let $\pr{\Gamma_n}_{n\geq 1}$ be a 
sequence of subsets of $\Z^d$ such that for each $n\geq 1$, $\Gamma_n$ is the disjoint union 
of $\Gamma_{n}\pr{w}, 1\leq w\leq J_n$ and 
\begin{equation}
 \Gamma_{n}\pr{w}=\Z^d\cap\prod_{q=1}^d \left[\underline{n}_q\pr{w,n},\overline{n}_q\pr{w,n}\right].
\end{equation}
Assume that $\ell_n:=\operatorname{Card}\pr{\Gamma_n}$ satisfies $\ell_{n+1}\geq \ell_n\geq \exp\pr{n^\delta}$,
$\sum_{k=1}^n\sqrt{\ell_k/LL\pr{\ell_k}}\leq C\ell_n/LL\pr{\ell_n}$, where $C>0$ and $\delta>0$ are independent of $n$, and 
$\overline{n}_q\pr{w,n}-\underline{n}_q\pr{w,n}\geq 4$ for each $n\geq 1$, $1\leq w\leq J_n$, $1\leq q\leq d$.

Then for all $1<p<2$, the following inequality holds:
\begin{multline} 
\norm{ 
\sup_{n\geq 1}\frac 1{ \sqrt{\ell_n LL\pr{\ell_n}  }    }
\abs{\sum_{\gri\in \Gamma_n }X_{\gr{i}}   }
}_p\\ 
\leq K\pr{p,d,C,\delta}\sum_{j\geq 0} \pr{j+1}^{d}L\pr{j}^{1/p}\pr{
\sum_{\gri\in\Z^d,\norm{\gri}_\infty=j}  \abs{a_{\gri }}^{2\gamma}}^{1/2}\norm{\eps_{\gr{0}}}_{2\gamma,0}.
\end{multline}
\end{Corollary}

\subsection{Functional of a Gaussian linear random field} 
Let $\pr{\eps_{\gri}}_{\gri \in \Z^d}$ be an i.i.d. random field where $\eps_{\gr{0}}$ has 
a standard normal distribution. Let 
\begin{equation}
Y_{\grj}:=\sum_{ \gri \in\Z^d  }a_{\gri}\eps_{\grj-\gri}, 
\end{equation}
where $a_{\gri}, \gri\in\Z^d$, are real numbers such that $\sum_{ \gri \in\Z^d  }a_{\gri}^2=1$.
Then $Y_{\grj}=f\pr{\pr{\eps_{\grj -\gri}}_{\gri\in\Z^d}}$ where 
$f\pr{\pr{x_{\gru}}_{\gru\in\Z^d}  }=\sum_{\gru\in\Z^d}a_{\gru}x_{\gru}$  if $\sum_{\gru\in\Z^d}a_{\gru}x_{\gru}$ converges 
(in the sense that $\lim_{m\to +\infty}\sum_{\substack{\gru\in\Z^d \\ \norm{\gru}_\infty\leq m } }a_{\gru}x_{\gru}$ exists)
and $0$ otherwise. Let $\eps'_{\gr{0}}$ be a random variable having the same law as $\eps_{\gr{0}}$ and 
independent of the random field $\pr{\eps_{\gri}}_{\gri \in \Z^d}$ and 
$Y_{\grj}^*=f\pr{\pr{\eps^*_{\grj -\gri}}_{\gri\in\Z^d}}$, where 
$\eps^*_{\gru}=\eps_{\gru}$ if $\gru\neq \gr{0}$ and $\eps^*_{\gr{0}}=\eps'_{\gr{0}}$.

Let $\Hca_0=\R$ and let $\Hca_1$ be the first Wiener chaos defined as the closed subspace of $\mathbb L^2$ 
generated by the random variables $Y_{\grj}$ and $Y^*_{\grj}$, $\grj\in\Z^d$. For $q\geq 2$, let 
$\Hca_q$ be the closed subspace of $\mathbb L^2$ generated by the random variables 
$H_q\pr{Y_{\grj}}$ and $H_q\pr{Y^*_{\grj}}$, $\grj\in\Z^d$, where $H_q$ is the 
$q$-th Hermite polynomial defined by 
\begin{equation}
H_q\pr{x}=\pr{-1}^q\exp\pr{\frac{x^2}2}\frac{d^q}{dx^q}\exp\pr{-\frac{x^2}2}.
\end{equation}

Given a function $f\colon \R\to\R$ such that $f\pr{Y_{\gr{0}}}\in\mathbb L^2$ and 
$\E{f\pr{Y_{\gr{0}}}}=0$, the following expansion holds: 
\begin{equation}
f\pr{Y_{\grj}}=\sum_{q=1}^{+\infty}c_q\pr{f}H_q\pr{Y_{\grj}},
\end{equation}
where 
\begin{equation}\label{eq:def_de_cqf}
c_q\pr{f}=\frac{1}{q!}\E{f\pr{Y_{\gr{0}}}H_q\pr{Y_{\gr{0}}}},
\end{equation}
provided that $\sum_{q=1}^{+\infty}q!c_q\pr{f}^2$ converges.

We are now in position to state the following consequence of Corollary~\ref{cor:coeff_phys_dep} 
for functionals of a Gaussian linear random field.

\begin{Corollary}\label{cor:fct_de_champs_Gaussiens_Hermite}
Let $\pr{\eps_{\gri}}_{\gri\in\Z^d}$ be an i.i.d. random field, 
\begin{equation}
Y_{\grj}:=\sum_{ \gri \in\Z^d  }a_{\gri}\eps_{\grj-\gri}, 
\end{equation}
where $a_{\gri}, \gri\in\Z^d$, are real numbers such that $\sum_{ \gri \in\Z^d  }a_{\gri}^2=1$. Let $f\colon \R\to \R$ be a function such that 
$\E{f\pr{Y_{\gr{0}}}}=0$ and 
\begin{equation}
C\pr{f}:=\sum_{q=1}^{+\infty}\sqrt{ q!}q^{d-\frac 12}\abs{c_q\pr{f}}<+\infty,
\end{equation}
where $c_q\pr{f}$ is defined by \eqref{eq:def_de_cqf}. Then for all $1<p<2$, 
\begin{multline}
\norm{ 
\sup_{\grn\in \N^d}\frac 1{ \sqrt{\abs{\grn} L L\pr{\abs{\grn}}
}    }
\abs{\sum_{\gr{1}\imd\gri \imd\grn  }f\pr{Y_{\gr{i}}   }}
}_p\\
\leq c_{p,d}\sum_{j\geq 0}\pr{j+1}^{d/2}\pr{\sum_{\gri\in\Z^d, \norm{\gri}_\infty=j }a_{\gri}^2  }^{1/2}  C\pr{f} ,
\end{multline}
where $c_{p,d}$ depends only on $p$ and $d$.
\end{Corollary}

\begin{Corollary}\label{cor:fct_de_champs_Gaussiens_Hermite_sum_on_rect}
Let $\pr{\eps_{\gri}}_{\gri\in\Z^d}$ be an i.i.d. random field, 
\begin{equation}
Y_{\grj}:=\sum_{ \gri \in\Z^d  }a_{\gri}\eps_{\grj-\gri}, 
\end{equation}
where $a_{\gri}, \gri\in\Z^d$, are real numbers such that $\sum_{ \gri \in\Z^d  }a_{\gri}^2=1$. Let $f\colon \R\to \R$ be a function such that 
$\E{f\pr{Y_{\gr{0}}}}=0$ and 
\begin{equation}
C\pr{f}:=\sum_{q=1}^{+\infty}\sqrt{ q!}q^{\frac 12}\abs{c_q\pr{f}}<+\infty,
\end{equation}
where $c_q\pr{f}$ is defined by \eqref{eq:def_de_cqf}.
Let $\pr{\Gamma_n}_{n\geq 1}$ be a 
sequence of subsets of $\Z^d$ such that for each $n\geq 1$, $\Gamma_n$ is the disjoint union 
of $\Gamma_{n}\pr{w}, 1\leq w\leq J_n$ and 
\begin{equation}
 \Gamma_{n}\pr{w}=\Z^d\cap\prod_{q=1}^d \left[\underline{n}_q\pr{w,n},\overline{n}_q\pr{w,n}\right].
\end{equation}
Assume that $\ell_n:=\operatorname{Card}\pr{\Gamma_n}$ satisfies $\ell_{n+1}\geq \ell_n\geq \exp\pr{n^\delta}$,
$\sum_{k=1}^n\sqrt{\ell_k/LL\pr{\ell_k}}\leq C\ell_n/LL\pr{\ell_n}$, where $C>0$ and $\delta>0$ are independent of $n$, and 
$\overline{n}_q\pr{w,n}-\underline{n}_q\pr{w,n}\geq 4$ for each $n\geq 1$, $1\leq w\leq J_n$, $1\leq q\leq d$.

Then for all $1<p<2$, the following inequality holds:
\begin{multline} 
\norm{ 
\sup_{n\geq 1}\frac 1{ \sqrt{\ell_n LL\pr{\ell_n}  }    }
\abs{\sum_{\gri\in \Gamma_n }f\pr{Y_{\gr{i}} }  }
}_p\\ 
\leq K\pr{p,d,C,\delta}C\pr{f}\sum_{j\geq 0} \pr{j+1}^{d}L\pr{j}^{1/p}\pr{
\sum_{\gri\in\Z^d,\norm{\gri}_\infty=j}  \abs{a_{\gri }}^{2}}^{1/2} .
\end{multline}
\end{Corollary}
 
\subsection{Volterra processes}

Volterra random fields of second order are defined in the following way. 
Let $\pr{\eps_{\gri}}_{\gri
\in\Z^d}$ be an i.i.d. collection of centered random variables and $\pr{a_{\gr{s_1},\gr{s_2}}
}_{\gr{s_1},\gr{s_2}\in\Z^d}$ be a family of real numbers such that $
a_{\gr{s_1},\gr{s_2}}=0$ if $\gr{s_1}=\gr{s_2}$ and $\sum_{\gr{s_1},\gr{s_2}\in\Z^d}a^2_{\gr{s_1},\gr{s_2}}$ is finite. Define 
\begin{equation}\label{eq:dfn_Volterra}
X_{\grj}:=\sum_{\gr{s_1},\gr{s_2}\in\Z^d}a_{\gr{s_1},\gr{s_2}}
\eps_{\grj-\gr{s_1}}\eps_{\grj-\gr{s_2}}.
\end{equation}

One can bound the term $\norm{X_{\gr{0},j  }}_{2,d-1}$, which leads to 
the following result.

\begin{Corollary}\label{cor:LLI_Volterra}
Let $\pr{X_{\gri}}_{\gri\in\Z^d}$ be a Volterra random field defined by 
\eqref{eq:dfn_Volterra}, where  $\pr{\eps_{\gri}}_{\gri
\in\Z^d}$ is i.i.d. and $\eps_{\gr{0}}\in \mathbb L_{2,d-1}$. Then for  
$1<p<2$, the following inequality holds:
\begin{multline}
\norm{ 
\sup_{n\geq 1}\frac 1{ \sqrt{\ell_n LL\pr{\ell_n}  }    }
\abs{\sum_{\gri\in \Gamma_n }X_{\gr{i}}   }
}_p\\
\leq c_{p,d}\sum_{j\geq 0} \pr{j+1}^{d/2} 
\pr{\sum_{\norm{\gr{s_1}}_\infty=j}
\sum_{\norm{\gr{s_2}}_\infty\leq j}
\pr{a_{\gr{s_1},\gr{s_2}}^2+ a_{\gr{s_2},\gr{s_1}}^2}        }^{1/2}\norm{\eps_{\gr{0}}}_{2,d-1}^2,
\end{multline}
where $c_{p,d}$ depends only on $p$ and $d$.
\end{Corollary}
 
\begin{Corollary}\label{cor:LLI_Volterra_sum_on_rect}
Let $\pr{X_{\gri}}_{\gri\in\Z^d}$ be a Volterra random field defined by 
\eqref{eq:dfn_Volterra}, where  $\pr{\eps_{\gri}}_{\gri
\in\Z^d}$ is i.i.d. and $\eps_{\gr{0}}\in \mathbb L^{2}$. Then for  
$1<p<2$, the following inequality holds:
\begin{multline}
\norm{ 
\sup_{\grn\in \N^d}\frac 1{ \sqrt{\abs{\grn} L L\pr{\abs{\grn}}}        }
\abs{\sum_{\gr{1}\imd\gri \imd\grn  }X_{\gr{i}}   }
}_p\\
\leq K\pr{p,d,C,\delta}\sum_{j\geq 0} \pr{j+1}^{d}L\pr{j}^{1/p} 
\pr{\sum_{\norm{\gr{s_1}}_\infty=j}
\sum_{\norm{\gr{s_2}}_\infty\leq j}
\pr{a_{\gr{s_1},\gr{s_2}}^2+ a_{\gr{s_2},\gr{s_1}}^2}        }^{1/2}\norm{\eps_{\gr{0}}}_{2,0}^2.
\end{multline}
\end{Corollary}
\section{Proofs}

\subsection{Tools for the proofs}

\subsubsection{Global ideas of proofs}

Let us explain the main steps in the proof of the results.
For Theorem~\ref{thm:cas_fonctionnelle_iid}, we proceed as follows.
\begin{enumerate}
\item As a first step, we prove Theorem~\ref{thm:cas_fonctionnelle_iid} in the 
i.i.d. case. To this aim, we show that the maximal function involved in 
Theorem~\ref{thm:cas_fonctionnelle_iid} can be replaced by another one where 
the supremum is taken only over elements of $\N^d$ having dyadic coordinates. 
Then we apply an appropriated exponential inequality in order to control the moments 
of this new maximal function. This gives Theorem~\ref{thm:iid}.
\item In order to prove Theorem~\ref{thm:cas_fonctionnelle_iid}, we bound the maximal 
function by a series of maximal functions associated to an i.i.d. random field. The contribution 
of these maximal functions can be estimated by using Theorem~\ref{thm:iid}.
\end{enumerate}
 
For Theorem~\ref{thm:linear_process_arbitrary_subsets}, we reduce the problem to the i.i.d. case.  
We use a truncation argument and the deviation 
inequality given in Proposition~\ref{prop:Freedman}. 

Theorem~\ref{thm:cas_fonctionnelle_union_rectangles} rests on Theorem~\ref{thm:linear_process_arbitrary_subsets} 
and a bound on the maximal function by maximal function of an i.i.d. random field. However, we need to 
reduce the summation on subsets of the union of rectangles and check that these sets satisfy the assumptions of 
Theorem~\ref{thm:linear_process_arbitrary_subsets} with the explicit constants.
 
\subsubsection{Weak $\mathbb L^p$-spaces}

The results of the paper provide a control of 
the $\el^p$ norm of a maximal function. However, 
it will sometimes be more convenient to work directly 
with tails. To this aim, we will consider weak
 $\el^p$-spaces.

\begin{Definition}
Let $p>1$. The weak $\el^p$-space, denoted by $\el^{p,w}$, is the 
space of random variables $X$ such that $\sup_{t>0}t^p\PP\ens{\abs X>t}$ 
is finite. 
\end{Definition}

These spaces can be endowed with a norm.

\begin{Lemma}\label{lem:lp_faibles}
Let $1<p< 2$. Define the following norm on $\el^{p,w}$
\begin{equation}
 \norm{X}_{p,w}:=\sup\ens{ \PP\pr{A}^{1/p-1}\E{\abs{X}\mathbf 1_A}   }.
\end{equation}

For all random variable $X\in \el^{p,w}$, the following inequality holds:
\begin{equation}\label{eq:comp_lp_weak}
c_p\norm{X}_{p,w}\leq \pr{\sup_{t>0}t^p\PP\ens{\abs X>t}  }^{1/p}
\leq C_p\norm{X}_{p,w}\leq C_p\norm{X}_{p},
\end{equation}
where $c_p$ and $C_p$ depend only on $p$.
\end{Lemma}

\subsubsection{Deviation inequalities}

The following deviation inequality is consequence of Theorem~2.1 in 
\cite{MR2462551}, which states that for a square integrable 
martingale differences sequence $\pr{d_j}_{j=1}^n$ with respect to a filtration $\pr{\Fca_j}_{j\geq 0}$, 
  \begin{equation}\label{eq:Bercu_Touati}
  \PP\pr{\ens{ \abs{ \sum_{j=1}^nd_j}>x}\cap \ens{\sum_{j=1}^n 
 \pr{ d_j^2 +\E{d_j^2\mid \Fca_{j-1}}}
   \leq y} 
  }\leq 2\exp\pr{-\frac{x^2}{2y}  }.
  \end{equation} 

 \begin{Proposition}\label{prop:inegalite_deviation_martingales}
  Let $\pr{d_j}_{j\geq 1}$ be a square integrable independent sequence. Then for all positive numbers $x$ and $y$, 
  the following inequality holds:
  \begin{equation}
  \PP\pr{\ens{ \abs{ \sum_{j=1}^nd_j}>x}\cap \ens{\sum_{j=1}^n 
  d_j^2 
   \leq y} 
  }\leq 2\exp\pr{-\frac{x^2}{2\pr{y+V^2}}  },
  \end{equation}   
  where $V^2=\sum_{j=1}^n\E{d_j^2}$.
 \end{Proposition}
 This follows from an application of \eqref{eq:Bercu_Touati} to $\Fca_i=\sigma\pr{d_j,1\leq j\leq i}$ and 
 $y$ replaced by $y+V^2$.

The following inequality will be needed in the proof of Theorem~\ref{thm:linear_process_arbitrary_subsets}; 
this is a version of Proposition~A.1 in \cite{MR2682267}, stated in this paper for martingale differences but presented here for independent sequences.
 
\begin{Proposition}\label{prop:Freedman}
Let $\pr{d_j}_{j\geq 1}$ be an independent centered sequence such that there exists a $c>0$ for which $\abs{d_j}\leq c$ almost surely. 
Let $n\geq 1$ be an integer and let $y$ be a real number such that $y\geq \sum_{i=1}^n\E{d_i^2}$. For all $x>0$, the 
inequality 
\begin{equation}
 \PP\ens{\abs{\sum_{i=1}^nd_i}>x }\leq 2\exp\pr{-\frac{y}{c^2} h\pr{\frac{xc}y}}
\end{equation}
takes place, where $h\pr{u}=\pr{1+u}\ln\pr{1+u}-u$.
\end{Proposition}

\begin{Lemma}\label{lem:Lemma_weak_type_estimate}
  Assume that $X$ and $Y$ are two 
  non-negative random variables such that for each positive $x$, 
  we have 
  \begin{equation}\label{eq:weak_type_assumption}
   x\PP\ens{X>x}\leqslant\mathbb 
  E\left[Y \mathbf 1\ens{X\geqslant x}\right].
  \end{equation}

  Then for each $t$, the following inequality holds:
  \begin{equation}
   \PP\ens{X>2t}\leqslant \int_1^{+\infty}\PP\ens{Y>st}\mathrm ds.
  \end{equation}
\end{Lemma} 
 
\begin{Lemma}\label{lem:lem_moment_tronques_esp_cond}
Let $X$ be an integrable non-negative random variable and let $\Gca$ be a sub-$\sigma$-algebra of 
$\Fca$. For all real number $y$, the following inequality holds:
\begin{equation}
\int_{1/2}^{+\infty}\PP\ens{\E{X\mid \Gca}> yu}\mathrm du\leq 
2\int_{1/2}^{+\infty}\PP\ens{X> yu}\mathrm du.
\end{equation} 
 
\end{Lemma} 
\begin{proof}
Replacing $X$ by $X/y$, there is no loss of generality by assuming that 
$y=1$. For any non-negative random variable $Z$, we have 
\begin{equation}
\E{Z\mathbf 1\ens{Z>1/2}}=\int_{1/2}^{\infty}\PP\ens{Z>t}\mathrm{d}t+
\frac 12\PP\ens{Z>1/2}.
\end{equation}
Therefore, 
\begin{equation}
\int_{1/2}^{+\infty}\PP\ens{\E{X\mid \Gca}> u}\mathrm du
\leq \E{\E{X\mid \Gca}  \mathbf 1\ens{\E{X\mid \Gca}>1/2}}
\end{equation}
and by definition of condition expectation, we get 
\begin{equation}\label{eq:lem_moment_tronques_esp_cond_etape_inter1}
\int_{1/2}^{+\infty}\PP\ens{\E{X\mid \Gca}> u}\mathrm du
\leq \E{X  \mathbf 1\ens{\E{X\mid \Gca}>1/2}}.
\end{equation}
 The last expectation can be written as 
 \begin{equation*}
 \int_{0}^{+\infty}
 \PP\pr{\ens{X >t}\cap \ens{\E{X\mid \Gca}>1/2}     }\mathrm dt
 \leq \frac 12\PP \ens{\E{X\mid \Gca}>1/2}  +\int_{1/2}^{+\infty}
     \PP \ens{X >t} \mathrm dt.
 \end{equation*}
 The first term of the right hand side does not exceed $1/2\int_{1/2}^{+\infty}\PP\ens{\E{X\mid \Gca}> u}\mathrm du$ hence 
 \begin{equation}\label{eq:lem_moment_tronques_esp_cond_etape_inter2}
 \E{X  \mathbf 1\ens{\E{X\mid \Gca}>1/2}}\leq 
 \frac 12\int_{1/2}^{+\infty}\PP\ens{\E{X\mid \Gca}> u}\mathrm du.
 \end{equation}
 We finish the proof by combining \eqref{eq:lem_moment_tronques_esp_cond_etape_inter1} 
 with \eqref{eq:lem_moment_tronques_esp_cond_etape_inter2}.
\end{proof}

 \begin{Proposition}\label{prop:extension_thm_ergodique_max}
 Let $\pr{Y_{\gri}}_{\gri\in\Z^d}$ be a strictly stationary random field
 such that each $Y_{\gri}$ is a non-negative random variable.
  Then for all positive $y$, 
 \begin{equation}
\PP\ens{\sup_{\grn \in \N^d}\frac 1{\abs{\grn}}
\sum_{\gr{1}\imd\gri\imd \grn }  Y_{\gri} 
>y  }
\leq   \int_{1}^{+\infty}\PP\ens{Y_{\gr{1}}>yu2^{-d}}
\pr{\log u}^{d-1}\mathrm du.
 \end{equation}
 \end{Proposition}
 \begin{proof}
 The proof is done by induction on the dimension. For $d=1$, this follows from a 
 combination of the maximal ergodic theorem 
 with Lemma~\ref{lem:Lemma_weak_type_estimate}.
Suppose now that for some $d\geq 2$, Proposition~\ref{prop:extension_thm_ergodique_max}
holds for all $\pr{d-1}$-dimensional random
fields and let  $\pr{Y_{\gri}}_{\gri\in\Z^d}$ be 
a strictly stationary random field
 such that each $Y_{\gri}$ is a non-negative 
 random variable. For $i_1,\dots,i_{d-1}\in \N$, 
 define 
 \begin{equation}
 \widetilde{Y}_{i_1,\dots,i_{d-1}}:=
 \sup_{n\geq 1}\frac 1n\sum_{i=1}^nY_{i_1,i_2,\dots,i_{d-1},i}.
 \end{equation}
Then 
\begin{equation}
\sup_{\grn \in \N^d}\frac 1{\abs{\grn}}
\sum_{\gr{1}\imd\gri\imd \grn }  Y_{\gri}\leq 
\sup_{\grn \in \N^{d-1}}\frac 1{\abs{\grn}}
\sum_{\gr{1}\imd\gri \imd\grn}\widetilde{Y_{\gri}}. 
\end{equation}
Applying the induction hypothesis to the strictly stationary random field 
$\pr{\widetilde{Y_{\gri}}}_{\gri\in \Z^{d-1}}$, we get 
 \begin{equation}
\PP\ens{\sup_{\grn \smd \gr{1}}\frac 1{\abs{\grn}}
\sum_{\gr{1}\imd\gri\imd \grn }  Y_{\gri} 
>y  }
\leq \int_{1}^{+\infty}
\PP\ens{ \widetilde{Y_{\gr{0}}} >yu2^{1-d}}\pr{\log u}^{d-2}\mathrm du.
 \end{equation}
 Applying the one dimensional case to the strictly stationary sequence 
 $\pr{Y_{0,\dots,0,i}}_{i\geq 1}$ gives 
 \begin{equation}
\PP\ens{\sup_{\grn \in\N^d}\frac 1{\abs{\grn}}
\sum_{\gr{1}\imd\gri\imd \grn }  Y_{\gri} 
>y  }
\leq  \int_{1}^{+\infty}\int_{1}^{+\infty}\PP\ens{  Y_{\gr{0}} >yuv2^{ -d}}\pr{\log u}^{d-2}\mathrm dv\mathrm du.
 \end{equation}
 
 and rearranging the integrals ends the proof.
 \end{proof}
 
\subsubsection{Facts on Orlicz spaces} 
 
 \begin{Lemma}\label{lem:norme_Orlicz_c_fois_fct_YOug}
 Let $p\geq 1$ and $r\geq 0$. Let $\varphi:=\varphi_{p,q}$ and 
 let $a>0$ be a constant. There exists a constant $c$ depending only on 
 $a$, $p$ and $q$ such that for all random variable $X$, 
 \begin{equation}
 \norm{X}_{\varphi}\leq c\norm{X}_{a\varphi}.
 \end{equation}
 \end{Lemma}
 
 \begin{proof}
 By homogeneity, it suffices to prove that for all random variable $X$ 
 such that $\norm{X}_{a\varphi}=1$, the equality $ \norm{X}_{\varphi}\leq 
 c$ holds for some $c$ depending only on 
 $a$, $p$ and $q$. Let $X$ be a random variable such that 
 $\norm{X}_{a\varphi}=1$. Then we know that $\E{
 a\varphi\pr{X}}=1$. Since there exists a constant $K$ such that 
 $\varphi\pr{uv}\leq K\varphi\pr{u}\varphi\pr{v}$, we derive that 
 \begin{equation*}
  \E{\varphi\pr{X/ c}}
  \leq \frac 1a aK\E{\varphi\pr{X}\varphi\pr{1/c}}
  = K\frac 1a\varphi\pr{1/b}.
 \end{equation*}
 We choose $c$ such that $K\frac 1a\varphi\pr{1/c}\leq 1$; for such 
 a $c$, inequality $\norm{X}_{\varphi}\leq c$ holds, which ends 
 the proof.
 \end{proof}

\begin{Lemma}\label{lem:norm_Orlicz_puissances}
Let $r\geq 0$. There exists a constant $c_r$ such that for any random variable $X$, 
\begin{equation}\label{eq:norm_Orlicz_carre}
\norm{X^2}_{1,r}\leq c_r\norm{X}_{2,r}^2;
\end{equation}
\begin{equation}\label{eq:norm_Orlicz_racine}
\norm{X^{1/2}}_{2,r}\leq c_r\norm{X}_{1,r}^{1/2}.
\end{equation}
\end{Lemma} 
\begin{proof}
We use the fact that there exists a constant $c$ depending only on 
$r$ such that for all positive $x$, 
\begin{equation}
 c^{-1}\varphi_{1,r}\pr{x^2}\leq 
\varphi_{2,r}\pr{x}\leq c\varphi_{1,r}\pr{x^2}
\end{equation}
 
Let us prove \eqref{eq:norm_Orlicz_carre}. By homogeneity, we assume that 
$\norm{X}_{2,r}=1$. Then  
\begin{equation}
\E{\varphi_{1,r}\pr{X^2 }}\leqslant c
\E{\varphi_{2,r}\pr{X }}\leq c,
\end{equation}
which shows that $\norm{X^2}_{ \varphi_{2,r}/c}\leq 1$
and we conclude by applying Lemma~\ref{lem:norme_Orlicz_c_fois_fct_YOug}.

The proof of \eqref{eq:norm_Orlicz_racine} follows exactly the 
same lines. 

\end{proof}

A tool will be the following version of Burkholder's inequality, which is a consequence 
of the combination of Lemmas~3.1 and 6.1 in \cite{MR0365692} (which is valid, 
since $\varphi_{2,r}$ satisfies $\sup_{t>0}\varphi_{2,r}\pr{2t}/\varphi_{2,r}\pr{t}<+\infty$).

\begin{Proposition}\label{prop:Burlholder_Orlicz}
Let $r\geq 0$ be a real number and let $\pr{d_k}_{k=1}^n$ be a martingale 
differences sequence with respect to the filtration $\pr{\Fca_k} _{k=0}^n$. 
The following inequality holds 
\begin{equation}
\norm{\sum_{k=1}^nd_k}_{2,r}\leq C_r \norm{\max_{1\leq k\leq n}\abs{d_k}}_{2,r}
+C_r\norm{\pr{\sum_{k=1}^n\E{d_k^2\mid \Fca_{k-1}}}^{1/2}}_{2,r}.
\end{equation} 
\end{Proposition}

Using Lemma~\ref{lem:norme_Orlicz_c_fois_fct_YOug}, we derive that 
\begin{multline}
\norm{\pr{\sum_{k=1}^n\E{d_k^2\mid \Fca_{k-1}}}^{1/2}}_{2,r}\leq 
c_r \norm{ \sum_{k=1}^n\E{d_k^2\mid \Fca_{k-1}}}_{1,r}^{1/2}\\
\leq c_r\pr{\sum_{k=1}^n \norm{ \E{d_k^2\mid \Fca_{k-1}}}_{1,r}}^{1/2}
\end{multline}
and the following corollary.
\begin{Corollary}\label{cor:Burlholder_Orlicz}
Let $r\geq 0$ be a real number and let $\pr{d_k}_{k=1}^n$ be a martingale 
differences sequence with respect to the filtration $\pr{\Fca_k} _{k=0}^n$. 
The following inequalities hold
\begin{equation}
\norm{\sum_{k=1}^nd_k}_{2,r}\leq C_r \norm{\max_{1\leq k\leq n}\abs{d_k}}_{2,r}
+C_r\pr{\sum_{k=1}^n \norm{ \E{d_k^2\mid \Fca_{k-1}}}_{1,r}}^{1/2};
\end{equation} 
\begin{equation}
\norm{\sum_{k=1}^nd_k}_{2,r}^2\leq K_r \sum_{k=1}^n\norm{d_k}_{2,r}^2 ,
\end{equation}
where $C_r$ and $K_r$ depend only on $r$.
\end{Corollary}

Observe also that 
\begin{equation}\label{eq:ineg_norm_1r_du_carre}
\norm{X^2}_{1,r}\leq c_r\norm{X}_{2,r}^2.
\end{equation}

\begin{Lemma}\label{lem:from_weak_type_to_Orlicz}
For all $p>1$ and $r\geq 0$, there exists a constant $c_{p,r}$ such that if 
 $X$ and $Y$ are two 
  non-negative random variables satisfying for each positive $x$, 
  \begin{equation} 
   x\PP\ens{X>x}\leqslant\mathbb 
  E\left[Y \mathbf 1\ens{X\geqslant x}\right],
  \end{equation}
  then $\norm{X}_{p,r}\leq c_{p,r}\norm{Y}_{p,r}$. 
\end{Lemma} 
 
 \begin{proof}
  This follows from a rewritting of $\E{\varphi_{p,r}\pr{X}}$
  using tails of $X$, that is, 
  \begin{equation}
   \E{\varphi_{p,r}\pr{X}}=\int_0^{+\infty}\varphi'_{p,r}\pr{t}
   \PP\ens{X>t}\mathrm dt.
  \end{equation}
  Since $\varphi'_{p,r}\pr{t}\leq ct^{p-1}\pr{1+\log\pr{1+t}}^r
  \leq C\varphi'_{p,r}\pr{t}$, we 
  derive by Lemma~\ref{lem:Lemma_weak_type_estimate} that 
  \begin{multline}
   \E{\varphi_{p,r}\pr{X}}\leq c
   \int_0^{+\infty} t^{p-1}\pr{1+\log\pr{1+t}}^r
\int_1^{+\infty}\PP\ens{Y>xt/2}\mathrm dx\mathrm dt\\
\leq cC\int_1^{+\infty}\E{\varphi_{p,r}\pr{2Y/x}
\mathbf 1\ens{2Y>x}}\mathrm dx.
  \end{multline}
  Since $\varphi_{p,r}\pr{uv}\leq K\varphi_{p,r}\pr{u}
  \varphi_{p,r}\pr{v}$ and the integral $\int_1^{+\infty}
  \varphi_{p,r}\pr{1/x}\mathrm{d}x$ converges, we proved that 
  $\E{\varphi_{p,r}\pr{X}}\leq c_{p,r}\E{\varphi_{p,r}\pr{Y}}$ and 
  we conclude using Lemma~\ref{lem:norme_Orlicz_c_fois_fct_YOug}.
 \end{proof}

 \subsection{Reduction to dyadics}
Let $d$ be a fixed integer and for $0\leq i\leq d-1$ define by  $\N_i$ the elements of 
$ \N^d$ whose coordinates $i+1,\dots,d$ are dyadic numbers. More 
formally, 
\begin{equation}
\N_i:=\ens{\gr{n}\in\N^d\mid  \mbox{ for all }
i+1\leq j\leq d, \exists k_j\in\N\cup\ens{0}\mbox{ such that }n_j=2^{k_j}   }.
\end{equation}
We also define $\N_d$ as $\N^d$. Notice that $\N_0$ is the set of all the elements of 
$\N^d$ such that all the coordinates are powers of $2$ of a non-negative integer.  The goal of this 
subsection is to show that it suffices to prove 
Theorem~\ref{thm:cas_fonctionnelle_iid}  where the supremum over $\N^d$ is replaced by 
the corresponding one over $\N_0$. 

\begin{Proposition}\label{prop:reduction_dyadiques_martingales_lex}
Let $\pr{X_{\gri}}_{\gri\in\Z^d}$ be  an i.i.d. centered random field. Then for all $1<p<2$, the following inequality holds 
\begin{equation}
\norm{\sup_{\gr{n}\in\N^d} 
\frac 1{\sqrt{\abs{\gr{n}}L L\pr{  \abs{\gr{n}     }}   }     }
\abs{\sum_{\gr{1}\imd \gr{i}\imd \gr{n}    }  X_{\gr{i}}   }
        }_{p}\leq c_{p,d}\norm{\sup_{\gr{n}\in\N_0} 
\frac 1{\sqrt{\abs{\gr{n}}L L\pr{  \abs{\gr{n}}     }   }     }
\abs{\sum_{\gr{1}\imd \gr{i}\imd \gr{n}    }  X_{\gr{i}}   }
        }_{p},
\end{equation}
where $c_{p,d}$ depends only on $d$.
\end{Proposition}

\begin{Lemma}\label{lem_Mi_Mi-1}
Let $\pr{a_{\gr{n}}}_{\gr{n}\in \N^d}$ be a family of positive numbers such that 
$a_{\gr{n}}\leq a_{\gr{n'}}$ if $\gr{n}\imd \gr{n'}$ and 
\begin{equation}\label{eq:definition_de_c}
c:=\sup_{\gr{n}\in\N^d}\max_{1\leq i\leq d}\frac{a_{\gr{n}+n_i \gr{e_i}   }}{a_{\gr{n}    }}<+\infty. 
\end{equation}
 Assume that $\pr{X_{\gr{i}}}_{\gr{i}\in\Z^d}$ 
is independent and centered.
Then for any real number number $x$ and any $i\in\ens{1,\dots,d}$, 
\begin{equation}\label{eq:lem_estimate_Mi}
\PP\ens{M_i>x}\leq \int_1^{+\infty}\PP\ens{M_{i-1}> \frac{ux}{2c}     }\mathrm du,
\end{equation}
 where 
 \begin{equation}
 M_i=\sup_{\gr{n}\in\N_i} 
\frac 1{\sqrt{\abs{\gr{n}}L L\pr{  \abs{\gr{n}}     }   }     }
\abs{\sum_{\gr{1}\imd \gr{i}\imd \gr{n}    }  X_{\gr{i}}   }.
 \end{equation}

\end{Lemma}

\begin{proof}
Let $0\leq i\leq d-1$. Define the random variables 
\begin{equation}
Y_N:= \frac1{a_{n_1,\dots,n_{i-1},N,n_{i+1},\dots,n_d}} \sup_{n_1,\dots,n_{i-1}}\sup_{n_{i+1},\dots,n_d}
\abs{S_{n_1,\dots,n_{i-1},N,n_{i+1},\dots,n_d}},
\end{equation}
\begin{equation}
Y'_N:=\frac{a_{n_1,\dots,n_{i-1},N,n_{i+1},\dots,n_d}}{a_{n_1,\dots,n_{i-1},2^{n+1},n_{i+1},\dots,n_d}}  Y_N,\quad  2^n+1\leq N\leq 2^{n+1}.
\end{equation}
and the following events 
\begin{equation}
A_N:=\ens{Y_N>x},  B_0=\emptyset,  B_N:=A_N\setminus \bigcup_{i=0}^{N-1}A_i,
\end{equation}
\begin{equation}\label{eq:definition_de_C_Nn}
C_{N,n}:=\begin{cases}
\bigcup_{i=2^n+1}^NB_i, &\mbox{ if }2^n+1\leq N\leq 2^{n+1};\\
\emptyset, &\mbox{ if } N\leq 2^n  \mbox{ or }  N>  2^{n+1}.
\end{cases}
\end{equation}
In this way, the set $\ens{M_i>x}$ can be expressed as the disjoint union 
$\bigcup_{N\geq 1}B_N$ hence 
\begin{equation}
\PP\ens{M_i>x}\leq \sum_{N\geq 1}\PP\pr{B_N}=
\sum_{n=0}^{+\infty}\sum_{N=2^n+1}^{2^{n+1}}\PP\pr{B_N}.
\end{equation}
Since $x\mathbf{1}\pr{B_N}\leq Y_N \mathbf{1}\pr{B_N}$, we infer that 
\begin{equation}
x\PP\ens{M_i>x}\leq\sum_{n=0}^{+\infty}\sum_{N=2^n+1}^{2^{n+1}}\E{ 
Y_N\mathbf 1\pr{B_N}
}.
\end{equation}
By definition of $c$ in \eqref{eq:definition_de_c}, we get that 
\begin{equation}\label{eq:estimate_tail_of_Mi}
x\PP\ens{M_i>x}\leq c\sum_{n=0}^{+\infty}\sum_{N=2^n+1}^{2^{n+1}}\E{ 
Y'_N\mathbf 1\pr{B_N}
}.
\end{equation}
Let $n\geq 0$ be fixed. Since $\mathbf 1\pr{B_N}=\mathbf 1\pr{C_{N,n}}-
\mathbf 1\pr{C_{N-1,n}}$ for all $n$ such that $2^n+1\leq N\leq 2^{n+1}$, 
\begin{align*}
\sum_{N=2^n+1}^{2^{n+1}}\E{ 
Y'_N\mathbf 1\pr{B_N}
}&=\sum_{N=2^n+1}^{2^{n+1}}\E{ 
Y'_N\pr{\mathbf 1\pr{C_{N,n}}-
\mathbf 1\pr{C_{N-1,n}}}
}\\
&=\E{\sum_{N=2^n+1}^{2^{n+1}} 
Y'_N \mathbf 1\pr{C_{N,n}}-
\sum_{N=2^n}^{2^{n+1}-1}Y'_{N+1} \mathbf 1\pr{C_{N,n}}}
\\
&\leq \E{Y'_{2^{n+1}} \mathbf 1\pr{C_{2^{n+1},n}} }+
\E{\sum_{N=2^n+1}^{2^{n+1}-1} 
\pr{ Y'_N  -Y'_{N+1} }     \mathbf 1\pr{C_{N,n}}}.
\end{align*}

The set $\mathbf 1\pr{C_{N,n}}$ is measurable with respect to the 
$\sigma$-algebra $\mathcal G_N$ generated by the random variables $\eps_{\gru}$, $\gru\in \Z^d$, $u_i\leq N$,
 and by independence of $\pr{X_{\gr{i}}}_{\gr{i}\in\Z^d}$ the random variable
 $\E{Y'_{N+1}-Y'_N\mid \Gca_N }$ is non-negative and consequently, 
 \begin{align}
 \E{\pr{ Y'_N  -Y'_{N+1} }     \mathbf 1\pr{C_{N,n}}}
 &=\E{\E{\pr{ Y'_N  -Y'_{N+1} }     \mathbf 1\pr{C_{N,n}} \mid \Gca_N}}\\
 &=\E{\mathbf 1\pr{C_{N,n}} \E{Y'_N  -  Y'_{N+1}    \mid \Gca_N }}\leq 0,
 \end{align}
 from which it follows that 
 \begin{equation}
 \sum_{N=2^n+1}^{2^{n+1}}\E{ 
Y'_N\mathbf 1\pr{B_N}
}\leq \E{Y'_{2^{n+1}} \mathbf 1\pr{C_{N,n}} }.
 \end{equation}
 The latter inequality combined with \eqref{eq:estimate_tail_of_Mi} allows to deduce that 
 \begin{equation}\label{eq:estimate_tail_of_Mi_bis}
x\PP\ens{M_i>x}\leq c\sum_{n=0}^{+\infty}\E{Y'_{2^{n+1}} \mathbf 1\pr{C_{2^{n+1},n}} }.
\end{equation}
Observe that for all $n\geq 0$, the random  variable $Y'_{2^{n+1}}$ is bounded by $M_{i-1}$. Combining this 
with the definition of $C_{N,n}$ given by \eqref{eq:definition_de_C_Nn}, we derive that 
\begin{equation}\label{eq:estimate_tail_of_Mi_ter}
x\PP\ens{M_i>x}\leq c\sum_{n=0}^{+\infty}\E{M_{i-1} \mathbf 1\pr{\bigcup_{k=2^n+1}^{2^{n+1}}B_k  } }.
\end{equation}
Since the family $\ens{B_k,k\geq 1}$ is pairwise disjoint, so is the family 
$\ens{\bigcup_{k=2^n+1}^{2^{n+1}}B_k ,n\geq 0}$. Therefore, using again the fact that 
$\ens{M_i>x}$ can be expressed as the disjoint union 
$\bigcup_{N\geq 1}B_N$, we establish the inequality 
\begin{equation}
x\PP\ens{M_i>x}\leq c\E{M_{i-1}\mathbf 1\ens{M_i>x}   }.
\end{equation}
We estimate the right hand side of the previous inequality in the 
following way:
\begin{align*}
 \E{M_{i-1}\mathbf 1\ens{M_i>x}   }&=\int_0^{+\infty} 
 \PP\pr{\ens{M_i>x}\cap \ens{M_{i-1}>t}}\mathrm dt\\
 &\leq \int_0^{x/\pr{2c}} 
 \PP\ens{M_i>x}\ \mathrm dt+\int_{x/\pr{2c}}^{+\infty} 
 \PP \ens{M_{i-1}>t}\mathrm dt\\
 &=\frac{x}{2c}\PP\ens{M_i>x}+\frac{x}{2c}
 \int_{1}^{+\infty} 
 \PP \ens{M_{i-1}>\frac{x}{2c}u}\mathrm du,
\end{align*}
from which \eqref{eq:lem_estimate_Mi} follows.
\end{proof}

\subsection{Proof of Theorem~\ref{thm:cas_fonctionnelle_iid} in the i.i.d. case}

In this Subsection, we will prove Theorem~\ref{thm:cas_fonctionnelle_iid} in the particular case
of an i.i.d. random field.

\begin{Theorem}\label{thm:iid}
Let $\pr{X_{\gr{i}}}_{\gr{i}\in \Z^d}$ be a centered i.i.d. random field. For all $1< p<2$, the following inequality holds:
\begin{equation}
\norm{ 
\sup_{\grn\in \N^d}\frac 1{ \sqrt{\abs{\grn} L L\pr{\abs{\grn}}   }    }
\abs{\sum_{\gr{1}\imd\gri \imd\grn  }X_{\gr{i}}   }
}_p\leq c_{p,d} \norm{X_{\gr{0}  }}_{2,d-1},
\end{equation}
where $c_{p,d}$ depends only on $p$ and $d$.
\end{Theorem}
\begin{proof}
We follows the ideas of the proof of Theorem~2.3 in \cite{MR3322323}.
Let us fix $1<p<2$. In view of Proposition~\ref{prop:reduction_dyadiques_martingales_lex} 
and Lemma~\ref{lem:Lemma_weak_type_estimate}, 
it suffices to establish that if $\norm{X_{\gr{0}}}_{2,d-1}=1$, then 
\begin{equation}\label{eq:bonre_martingale_lex_reduite}
x^p\PP\ens{\sup_{\gr{n}\in\N_0} 
\frac 1{\sqrt{\abs{\gr{n}}L L\pr{  \abs{\gr{n}     }}   }     }
\abs{\sum_{\gr{1}\imd \gr{i}\imd \gr{n}   }  X_{\gr{i}}   }
    >x    }   \leq C_{p,d} ,
\end{equation}
It suffices to prove \eqref{eq:bonre_martingale_lex_reduite} for $x$ such that 
\begin{equation}\label{eq:hypothese_sur_x}
\frac{x^2}{2x^p+2}-d\geq\frac{pd}{2-p}.
\end{equation}

 Let us fix such an $x$.
We will use the notation $s\pr{\grn}= 
 \sum_{q=1}^dn_q$ for $\grn\in \N^d$.
 Define the events 
\begin{equation}
A_{\grn}:=\ens{\frac 1{\sqrt{2^{s\pr{\gr{n}}}L \pr{  s\pr{\gr{n}     }}   }     }
\abs{\sum_{\gr{1}\imd \gr{i}\imd 2^{\gr{n} }   }  X_{\gr{i}}   }
    >x    };
\end{equation}
 \begin{equation}
 B_{\grn}:=\ens{\frac 1{ 2^{s\pr{\gr{n}}}       }
\abs{\sum_{\gr{1}\imd \gr{i}\imd 2^{\gr{n} }  }  X_{\gr{i}}^2   }
    >x^p    }.
 \end{equation}
 
 It suffices to establish that 
 \begin{equation}\label{eq:reduction_demo_LIL_martingales_ordre_lexico}
 \sum_{\grn\smd \gr{0}}  \PP\pr{A_{\grn}\cap  B_{\grn}^c }+
 \PP\pr{\bigcup_{\grn\smd\gr{0}}B_{\gr{n}} }\leq C_{p,d}x^{-p}.
 \end{equation}
 Let us fix an integer $N$. Define the sets 
 $J:= \ens{\grn\smd \gr{0}\mid s\pr{\grn} \leq N  }$ and 
 $J':= \ens{\grn\smd \gr{0}\mid s\pr{\grn} > N  }$, where $s\pr{\grn}= 
 \sum_{q=1}^dn_q$. Using Chebyshev's inequality, 
 we get 
 \begin{equation}
  \sum_{\grn\in J}  \PP\pr{A_{\grn}\cap  B_{\grn}^c }\leq 
  x^{-2}\operatorname{Card}\pr{J}\leq  x^{-2}c_d\sum_{k=1}^Nk^{d-1}\leq 
  c_d  x^{-2}N^d.
 \end{equation}
 Now, we control for a fixed $\grn\in J'$ the quantity $\PP\pr{A_{\grn}\cap  B_{\grn}^c }$. 
Let $\pr{d_i}_{i\in \Z}$ be an i.i.d. sequence such that $d_1$ has the same distribution as 
$X_{\gr{1}}$. Then 
\begin{equation}
\PP\pr{A_{\grn}\cap  B_{\grn}^c }=\PP\pr{\ens{\frac 1{\sqrt{2^{s\pr{\gr{n}}}L \pr{  s\pr{\gr{n}     }}   }     }
\abs{\sum_{j=1}^{2^{s\pr{\gr{n}}}   }    d_j   }  
    >x    } 
\cap     \ens{\frac 1{ 2^{s\pr{\gr{n}}}       }
\sum_{j=1}^{2^{s\pr{\gr{n}}}   }    d_j^2    
    \leq x^p    } 
     }
\end{equation}
 and by using Proposition~\ref{prop:inegalite_deviation_martingales} in the following context:
 $n:=2^{s\pr{\grn}}$  and $y=x^p2^{s\pr{\gr{n}}}$,
 we obtain the bound 
 \begin{equation}
 \PP\pr{A_{\grn}\cap  B_{\grn}^c }\leq 2\exp\pr{- 2^{s\pr{\gr{n}}}L\pr{s\pr{\grn}  }\frac{  x^{2}}{2\pr{x^p2^{s\pr{\gr{n}}}
 + 2^{s\pr{\gr{n}}}\E{X_{\gr{1}}^2   }}  }   }.
 \end{equation}
 Due to the assumption \eqref{eq:hypothese_sur_x} and the fact that $\E{X_{\gr{1}}^2   }\leq 1$, we obtain the estimate
 \begin{equation}
 \PP\pr{A_{\grn}\cap  B_{\grn}^c }\leq 2\exp\pr{- L\pr{s\pr{\grn}  }\pr{\frac{pd}{2-p} +d}    }.
 \end{equation}
 As the number of elements $\grn$ of $\N^d$ such that $s\pr{\grn}=k$ is of order 
 $k^{d-1}$ and $x\geq a_{p,d}$, we get that 
 \begin{equation}
  \PP\pr{A_{\grn}\cap  B_{\grn}^c }\leq c_d\sum_{k\geq N+1}k^{d-1-\pr{\frac{pd}{2-p} +d}}
  \leq c_d   \sum_{k\geq N+1}k^{-1-\frac{pd}{2-p}}\leq 
c_{p,d}N^{-\frac{pd}{2-p}}.
 \end{equation}
Using Proposition~\ref{prop:extension_thm_ergodique_max}, we derive that 
\begin{equation}
 \PP\pr{\bigcup_{\grn\smd\gr{0}}B_{\gr{n}} }\leq c_dx^{-p}
 \E{\varphi_{2,d-1}\pr{\abs{X_{\gr{0}}  }}}\leq c_dx^{-p}.
\end{equation}
Choosing $N:=\lfloor x^{\frac{2-p}d}\rfloor$, the previous estimations 
give \eqref{eq:reduction_demo_LIL_martingales_ordre_lexico}.
This ends the proof of Theorem~\ref{thm:iid}.
\end{proof}

\subsection{Proof of Theorem~\ref{thm:cas_fonctionnelle_iid}}

For $\grn\in \N^d$, define $c_{\grn}:= \sqrt{\abs{\gr{n}}
L L\pr{\abs{\grn}    }}  $.
We associate to a random field $\pr{X_{\gri}}_{\gri\in\Z^d}$ a maximal function, namely 
\begin{equation}
M\pr{\pr{X_{\gri}}_{\gri\in\Z^d}    }:=
\sup_{\grn\in\N^d}\frac 1{c_{\grn}}\abs{\sum_{\gr{1}\imd\gri \imd \grn }  
X_{\gri}
}.
\end{equation}

By the martingale convergence theorem, the decomposition 
$X_{\gri}=X_{\gri,0}+\sum_{j=1}^{+\infty}X_{\gri,j}$ holds almost surely where 
$X_{\gri,j}=\E{X_{\gr{i}}\mid \sigma\ens{\eps_{\gr{u}},\norm{\gr{u}-\gri}_\infty\leq 
 j   }}-\E{X_{\gr{i}}\mid \sigma\ens{\eps_{\gr{u}},\norm{\gr{u}-\gri }_\infty\leq 
 j -1  }}$ hence 
\begin{equation}
M\pr{\pr{X_{\gri}}_{\gri\in\Z^d}    }\leq 
M\pr{\pr{X_{\gri,0}}_{\gri\in\Z^d}    }+\sum_{j=1}^{+\infty}
M\pr{\pr{X_{\gri,j}}_{\gri\in\Z^d}    }.
\end{equation}
Therefore, it suffices to control the maximal function associated to 
each random field $\pr{X_{\gri,j}}_{\gri\in\Z^d} $, which has the feature that for all fixed $j$,  
$X_{\gri,j}$ is a function of $\pr{\eps_{\gri-\gru}}_{\norm{\gru}_\infty\leq j   }$. For such a 
random field, the maximal function can be bounded by a finite sum of maximal functions associated 
to an i.i.d. random field.

Let $\kappa$ be a constant such that for all $\gr{N}\smd \gr{1}$, $c_{\gr{N}+\gr{1}}\leq \kappa c_{\gr{N}}$.
\begin{Lemma}\label{lem:etape_inter_chamP-m_dep}
For all positive integer $k_0$, the following inequality holds:
\begin{equation}\label{eq:etape_inter_chamP-m_dep}
M\pr{\pr{X_{\gri}}_{\gri\in\Z^d}    }\leq \kappa k_0^{-d/2}\sum_{\gr{1}
\imd \gr{b}\imd  k_0 \gr{1}}  \widetilde{M}\pr{\pr{X_{k_0\gra+\grb}}_{\gra\in\Z^d}    },
\end{equation}
where 
\begin{equation}
\widetilde{M}\pr{\pr{X_{\gri}}_{\gri\in\Z^d}    }:=
\sup_{\grn\in\N^d}\frac 1{c_{\grn}}\abs{\sum_{\gr{0}\imd\gri \imd \grn }  
X_{\gri}
}.
\end{equation}
\end{Lemma}
Observe that the only difference between $M$ and $\widetilde{M}$ lies in the index of summation. 
If $\pr{X_{\gri}}_{\gri\in\Z^d}$ is strictly stationary, then 
\begin{equation}\label{eq:lien_M_tilde_M}
\norm{\widetilde{M}\pr{\pr{X_{\gri}}_{\gri\in\Z^d}    }}_p=\norm{
\sup_{\grn\in\N^d}\frac 1{c_{\grn}}\abs{\sum_{\gr{1}\imd\gri \imd \grn +\gr{1} }  
X_{\gri}
}}_p\leq \kappa \norm{M\pr{\pr{X_{\gri}}_{\gri\in\Z^d}    }}_p.
\end{equation}
\begin{proof}
Let 
\begin{equation}
M'\pr{\pr{X_{\gri}}_{\gri\in\Z^d}    }:= \sup_{\gr{N}\smd\gr{1}}\frac 1{c_{k_0\gr{N}}}\abs{\sum_{\gr{1}\imd\gri \imd k_0\gr{N} }  
X_{\gri}
}
\end{equation}
If $M$ was replaced by $M'$ in the left hand side of \eqref{eq:etape_inter_chamP-m_dep}, then Lemma~\ref{lem:etape_inter_chamP-m_dep} 
would follow from a decomposition of $\sum_{\gr{1}\imd\gri \imd \grn }  
X_{\gri}$ according to the remainder of the Euclidean division by $k_0$ of each $i_q$. We will 
link $M$ and $M'$. We start from the estimate
\begin{equation}
M\pr{\pr{X_{\gri}}_{\gri\in\Z^d}    }\leq 
\sup_{\gr{N}\smd\gr{0}}\max_{\gr{1}\imd \gr{b}\imd k_0\gr{1}} 
\frac 1{c_{k_0\gr{N}+\grb}}\max_{\gr{1}\imd \grn \imd k_0\gr{N}+\grb}
\abs{\sum_{\gr{1}\imd\gri \imd \grn }  
X_{\gri}}.
\end{equation}
Using the fact that $c_{k_0\gr{N}+\grb}\geq c_{k_0\gr{N} }\geq k_0^{d/2}c_{\gr{N} }$
and the inequality 
\begin{equation}
\max_{\gr{1}\imd \gr{b}\imd k_0\gr{1}} \max_{\gr{1}\imd \grn \imd k_0\gr{N}+\grb}\abs{\sum_{\gr{1}\imd\gri \imd \grn }  
X_{\gri}}
\leq \max_{\gr{1}\imd \grn \imd k_0\pr{\gr{N}+\gr{1} }}\abs{\sum_{\gr{1}\imd\gri \imd \grn }  
X_{\gri}},
\end{equation}
we infer that 
\begin{equation}
M\pr{\pr{X_{\gri}}_{\gri\in\Z^d}    }\leq k_0^{-d/2}\sup_{\gr{N}\smd\gr{0}} 
\frac 1{c_{\gr{N}}}    \max_{\gr{1}\imd \grn \imd k_0\pr{\gr{N}+\gr{1} }}\abs{\sum_{\gr{1}\imd\gri \imd \grn }  
X_{\gri}} .
\end{equation}
From the definition of $\kappa$, we get 
\begin{equation}
M\pr{\pr{X_{\gri}}_{\gri\in\Z^d}    }\leq \kappa k_0^{-d/2}\sup_{\gr{N}\smd\gr{1}} 
\frac 1{c_{\gr{N}}}    \max_{\gr{1}\imd \grn \imd k_0 \gr{N}  }\abs{\sum_{\gr{1}\imd\gri \imd \grn }  
X_{\gri}} .
\end{equation}
Using the fact that for all $\grn$ such that $\gr{1}\imd \grn \imd k_0 \gr{N}$, 
\begin{equation}
\abs{\sum_{\gr{1}\imd\gri \imd \grn }  X_{\gri}}\leq \max_{\gr{1}\imd \gr{n'}\imd \gr{N}}
\sum_{\gr{1}\imd\grb\imd k_0\gr{1}}\abs{\sum_{\gr{0}\imd \gr{a}\imd \gr{n'}   }    X_{k_0\gr{a}+\grb} },
\end{equation}
we conclude the proof of Lemma~\ref{lem:etape_inter_chamP-m_dep}.
\end{proof}

For all fixed $j\geq 1$, we apply Lemma~\ref{lem:etape_inter_chamP-m_dep} to 
$\widetilde{X_{\gri}}:= X_{\gri,j}$ and $k_0:= 4j+1$ in order to obtain, after having 
taken the $\mathbb L^p$-norms, 
\begin{multline}
\norm{M\pr{\pr{X_{\gri}}_{\gri\in\Z^d}    }}_p\leq 
\norm{M\pr{\pr{X_{\gri,0}}_{\gri\in\Z^d}    }}_p\\
+\kappa\sum_{j=1}^{+\infty}
\pr{4j+1}^{-d/2}\sum_{\gr{1}
\imd \gr{b}\imd \pr{4j+1}\gr{1}}  
\norm{\widetilde{M}\pr{\pr{X_{\pr{4j+1}\gra+\grb}}_{\gra\in\Z^d}    }}_p.
\end{multline}
We then conclude by applying first \eqref{eq:lien_M_tilde_M} then Theorem~\ref{thm:iid} 
for all $j\geq 1$ and $\grb\in\Z^d$ such that $\gr{1}
\imd \gr{b}\imd \pr{4j+1}\gr{1}$ to the i.i.d. random field $\pr{X_{\pr{4j+1}\gra+\grb}}_{\gra\in\Z^d}$, and 
to the i.i.d. random field $\pr{X_{\gri,0}}_{\gri\in\Z^d}$.

This ends the proof of Theorem~\ref{thm:cas_fonctionnelle_iid}.

\subsection{Proof of Theorem~\ref{thm:linear_process_arbitrary_subsets}}

First observe that by definition of linear processes, 
\begin{equation}
 \sup_{n\geq 1}\frac 1{ \sqrt{\ell_n LL\pr{\ell_n}  }    }
\abs{\sum_{\gri\in \Lambda_n }X_{\gr{i}}   }
\leq \sum_{\grj\in\Z^d}\sup_{n\geq 1}\frac 1{ \sqrt{\ell_n LL\pr{\ell_n}  }    }
\abs{\sum_{\gri\in \Lambda_n }a_{\grj}\eps_{\gr{i}-\grj}   },
\end{equation}
hence 
\begin{equation}
 \norm{\sup_{n\geq 1}\frac 1{ \sqrt{\ell_n LL\pr{\ell_n}  }    }
\abs{\sum_{\gri\in \Lambda_n }X_{\gr{i}}   }}_p
\leq \sum_{\grj\in\Z^d}\abs{a_{\grj}}\norm{\sup_{n\geq 1}\frac 1{ \sqrt{\ell_n LL\pr{\ell_n}  }    }
\abs{\sum_{\gri\in \Lambda_n-\grj }\eps_{\gr{i}}   }}_p,
\end{equation}
where $\Lambda_n-\grj$ is the set of the elements of $\Z^d$ of the form $\gri-\grj$, where 
$\gri\in\Lambda_n$. Since $\Lambda_n-\grj$ has the same number of elements as $\Lambda_n$, it 
suffices to prove Theorem~\ref{thm:linear_process_arbitrary_subsets} in the i.i.d. case.

We thus assume that the random field $\pr{X_{\gri}}_{\gri\in\Z^d}$ is i.i.d., centered and 
$\E{X_{\gr{1}}^2}=1$.
We will essentially follow the steps of the proof Theorem~2.3 in \cite{MR3322323}; nervertheless, Theorem~\ref{thm:linear_process_arbitrary_subsets} 
is not a direct consequence of this result as the sets $\Lambda_n$ may not satisfy any inclusion relation.
As in the proof of Theorem~\ref{thm:cas_fonctionnelle_iid}, it suffices to prove that for each $1<p<2$, there exists some 
positive constants $t_0$ and $c_p$ such that 
\begin{equation}\label{eq:goal_iid_sum_on_sets}
 \sup_{t\geq t_0}t^p\PP\ens{\sup_{n\geq 1}\frac 1{ \sqrt{\ell_n LL\pr{\ell_n}  }    }
\abs{\sum_{\gri\in \Lambda_n }X_{\gr{i}}   }>t}\leq c_pC\delta^{-p/2} .
\end{equation}
We will take 
\begin{equation}
 t_0:=\sqrt{\frac{2}{\pr{2-p}h\pr{1}\delta}},
\end{equation}
where $h$ is like in Proposition~\ref{prop:Freedman}. This will give that 
$\norm{\sup_{n\geq 1}\frac 1{ \sqrt{\ell_n LL\pr{\ell_n}  }    }
\abs{\sum_{\gri\in \Lambda_n }X_{\gr{i}}   }}_{p,w}\leq \pr{c_pC\delta^{-p/2}}^{1/p}+t_0$ and using 
\eqref{eq:comp_lp_weak} with $p+\eps$ will give \eqref{eq:control_fct_max_somme_sur_ens_processus_lineaire}.

In order to show \eqref{eq:goal_iid_sum_on_sets}, we use truncation and introduce the random variables 
\begin{equation}
 X_{\gri,n}:=X_{\gri}\mathbf{1}\ens{\abs{X_{\gri}}\leq at \sqrt{\frac{\ell_n }{LL\pr{\ell_n}  }    }  }
 -\E{X_{\gri}\mathbf{1}\ens{\abs{X_{\gri}}\leq at\sqrt{\frac{\ell_n }{LL\pr{\ell_n}  }    }  }},
\end{equation}
\begin{equation}
 X'_{\gri,n}:=X_{\gri}\mathbf{1}\ens{\abs{X_{\gri}}> at\sqrt{\frac{\ell_n }{LL\pr{\ell_n}  }    }  }
 -\E{X_{\gri}\mathbf{1}\ens{\abs{X_{\gri}}> at\sqrt{\frac{\ell_n }{LL\pr{\ell_n}  }    }  }},
\end{equation}
where $t\geq t_0$ is fixed and $a=\pr{2-p}h\pr{1}\delta/2$, where $h$ is like in Proposition~\ref{prop:Freedman}. Let $N:=\lfloor t^{2-p}\rfloor+1$. Then 

\begin{multline}
 \PP\ens{\sup_{n\geq 1}\frac 1{ \sqrt{\ell_n LL\pr{\ell_n}  }    }
\abs{\sum_{\gri\in \Lambda_n }X_{\gr{i}}   }>2t}\leq 
\sum_{n=1}^N\PP\ens{\frac 1{ \sqrt{\ell_n LL\pr{\ell_n}  }    }
\abs{\sum_{\gri\in \Lambda_n }X_{\gr{i}}   }>2t}\\
+\sum_{n\geq N+1}\PP\ens{\frac 1{ \sqrt{\ell_n LL\pr{\ell_n}  }    }
\abs{\sum_{\gri\in \Lambda_n }X_{\gr{i},n}   }>t}+
\sum_{n\geq N+1}\PP\ens{\frac 1{ \sqrt{\ell_n LL\pr{\ell_n}  }    }
\abs{\sum_{\gri\in \Lambda_n }X'_{\gr{i},n}   }>t}.
\end{multline}
The term $\sum_{n=1}^N\PP\ens{\frac 1{ \sqrt{\ell_n LL\pr{\ell_n}  }    }
\abs{\sum_{\gri\in \Lambda_n }X_{\gr{i}}   }>2t}$ can be controlled by 
Chebyshev's inequality, independence and the assumption of unit variance of $X_{\gri}$. We get that 
\begin{equation}
 \sum_{n=1}^N\PP\ens{\frac 1{ \sqrt{\ell_n LL\pr{\ell_n}  }    }
\abs{\sum_{\gri\in \Lambda_n }X_{\gr{i}}   }>2t}\leq \frac{1}{4t^2}N\leq 
\frac 14 t^{-p}+\frac{1}{4t^{-p}t_0^{2-p}}.
\end{equation}
In order to control $\PP\ens{\frac 1{ \sqrt{\ell_n LL\pr{\ell_n}  }    }
\abs{\sum_{\gri\in \Lambda_n }X_{\gr{i},n}   }>t}$ for $n\geq N+1$, we apply 
Proposition~\ref{prop:Freedman} in the following setting: $c=\frac{2at\sqrt{\ell_n}}{\sqrt{ LL\pr{\ell_n}  }    }$, 
$x=t\sqrt{\ell_n LL\pr{\ell_n}  }  $ and 
$y= t^2a\ell_n$. We obtain that 
\begin{equation}
 \sum_{n\geq N+1}\PP\ens{\frac 1{ \sqrt{\ell_n LL\pr{\ell_n}  }    }
\abs{\sum_{\gri\in \Lambda_n }X_{\gr{i},n}   }>t}
\leq 2 \sum_{n\geq N+1}\exp\pr{-\frac{h\pr{1}}{a}LL\pr{\ell_n  }}.
\end{equation}
From the assumption $\ell_n\geq \exp\pr{n^\delta}$, we get that 
\begin{equation}
  \sum_{n\geq N+1}\PP\ens{\frac 1{ \sqrt{\ell_n LL\pr{\ell_n}  }    }
\abs{\sum_{\gri\in \Lambda_n }X_{\gr{i},n}   }>t}
\leq 2 \sum_{n\geq N+1}\exp\pr{-\frac{h\pr{1}\delta}{a}L \pr{n  }}.
\end{equation}

One gets that the contribution 
of the previous term does not exceed $4 t^{-p}/\pr{2-p}$. It remains to control 
$\sum_{n\geq N+1}\PP\ens{\frac 1{ \sqrt{\ell_n LL\pr{\ell_n}  }    }
\abs{\sum_{\gri\in \Lambda_n }X'_{\gr{i},n}   }>t}$. To do so, we start from Markov's inequality:
\begin{equation}
 \sum_{n\geq N+1}\PP\ens{\frac 1{ \sqrt{\ell_n LL\pr{\ell_n}  }    }
\abs{\sum_{\gri\in \Lambda_n }X'_{\gr{i},n}   }>t}\leq 
\sum_{n\geq N+1}\frac{1}{t\sqrt{\ell_n LL\pr{\ell_n}  } }\sum_{\gri\in \Lambda_n }
\E{\abs{X'_{\gr{i},n}   }}
\end{equation}
and by definition of $X'_{\gr{i},n}$, we get that 
\begin{equation}
 \sum_{n\geq N+1}\PP\ens{\frac 1{ \sqrt{\ell_n LL\pr{\ell_n}  }    }
\abs{\sum_{\gri\in \Lambda_n }X'_{\gr{i},n}   }>t}\leq \frac{2}t 
\sum_{n\geq 1}b_n
\E{\abs{X_{\gr{1}  }}\mathbf{1}\ens{\abs{X_{\gr{1}}}>  atb_n   }  },
\end{equation}
where $b_n:= \frac{2^{n/2}}{\sqrt{L\pr{n}} }$. For a random variable $X$, 
\begin{equation}
 \sum_{n\geq 1}b_n\mathbf{1}\ens{\abs{X}>b_n}= \sum_{n\geq 1}b_n\sum_{k\geq n}\mathbf{1}\ens{b_k<\abs{X}\leq b_{k+1}}
 \leq C\abs{X},
\end{equation}
as $\sum_{n=1}^kb_n\leq Cb_k$. Applying this to $X=\abs{X_{\gr{1}}}/\pr{at}$ gives that 
\begin{equation}
 \sum_{n\geq N+1}\PP\ens{\frac 1{ \sqrt{\ell_n LL\pr{\ell_n}  }    }
\abs{\sum_{\gri\in \Lambda_n }X'_{\gr{i},n}   }>t}\leq \frac{4C}{\pr{2-p}\delta h\pr{1}t^pt_0^{2-p}} 
\end{equation}
and we get \eqref{eq:goal_iid_sum_on_sets} in view of the previous estimates. This ends the proof of Theorem~\ref{thm:linear_process_arbitrary_subsets}.

\subsection{Proof of Theorem~\ref{thm:cas_fonctionnelle_union_rectangles}}

From the result of the previous Subsection, we got that for an i.i.d. centered random field $\pr{Y_{\gri}}_{\gri\in\Z^d}$ 
and subsets $\Lambda_n$ of $\Z^d$ of cardinal $\ell_n$
\begin{equation}\label{eq:control_fct_max_somme_sur_ens_iid}
\norm{ 
\sup_{n\geq 1}\frac 1{ \sqrt{\ell_n LL\pr{\ell_n}  }    }
\abs{\sum_{\gri\in \Lambda_n }Y_{\gr{i}}   }
}_p\leq C_p C^{1/p}\delta^{-1/2} \norm{Y_{\gr{0}}}_2,
\end{equation}
provided that $\ell_{n+1}\geq \ell_n\geq \exp\pr{n^\delta}$ and
$\sum_{k=1}^n\sqrt{\ell_k/LL\pr{\ell_k}}\leq C\ell_n/LL\pr{\ell_n}$.
By the martingale convergence theorem, the decomposition 
$X_{\gri}=X_{\gri,0}+\sum_{j=1}^{+\infty}X_{\gri,j}$ holds almost surely where 
$X_{\gri,j}=\E{X_{\gr{i}}\mid \sigma\ens{\eps_{\gr{u}},\norm{\gr{u}-\gri}_\infty\leq 
 j   }}-\E{X_{\gr{i}}\mid \sigma\ens{\eps_{\gr{u}},\norm{\gr{u}-\gri }_\infty\leq 
 j -1  }}$ hence 
\begin{equation}\label{eq:control_fct_max_somme_sur_union_rect}
\norm{ 
\sup_{n\geq 1}\frac 1{ \sqrt{\ell_n LL\pr{\ell_n}  }    }
\abs{\sum_{\gri\in \Lambda_n }X_{\gr{i}}   }
}_p\leq \sum_{j\geq 0}\norm{ 
\sup_{n\geq 1}\frac 1{ \sqrt{\ell_n LL\pr{\ell_n}  }    }
\abs{\sum_{\gri\in \Lambda_n }X_{\gr{i},j}   }
}_p.
\end{equation}
For each fixed $j\geq 1$, $X_{\gr{i},j}$ is a function of the random variables 
$\eps_{\gru}$, where $\gru$ runs over the $\gru\in\Z^d$ such that 
$\norm{\gru-\gri}\leq j$. Although $\pr{X_{\gr{i},j}}_{\gri\in\Z^d}$ is 
not an independent random field, one can cut the summation over $\Gamma_n$ 
according to the remainder with respect to the Euclidean division 
of each coordinate by $4j+2$. Indeed, for $j\geq 1$ fixed, define the sets 
\begin{equation}
 \Gamma_n^{\gra,j}:=\ens{\gri\in \Z^d, \pr{4j+2}\gr{1}+\gra\in \Gamma_n}, \gra=\pr{a_q}_{q=1}^d, 
 0\leq a_q\leq 4j+1
\end{equation}
and the random fields $\pr{Y_{\gri,j}^{\gra}}_{\gri\in\Z}$ by 
\begin{equation}
 Y_{\gri,j}^{\gra}=X_{\pr{4j+2}\gr{i}+\gra ,j}.
\end{equation}

Then the following inequality holds:
\begin{equation}\label{eq:control_fct_max_somme_sur_union_rect_2}
 \norm{ 
\sup_{n\geq 1}\frac 1{ \sqrt{\ell_n LL\pr{\ell_n}  }    }
\abs{\sum_{\gri\in \Gamma_n }X_{\gr{i},j}   }
}_p\leq \sum_{\gr{0}\imd\gra \imd \pr{4j+1}\gr{1}}
 \norm{ 
\sup_{n\geq 1}\frac 1{ \sqrt{\ell_n LL\pr{\ell_n}  }    }
\abs{\sum_{\gri\in \Gamma_n^{\gra} }Y^{\gra}_{\gr{i},j}   }
}_p.
\end{equation}

Now, we have to show that the sets $\Gamma_n^{\gra,j}$ also satisfy the assumptions 
$\ell^{\gra,j}_{n+1}\geq \ell^{\gra,j}_n \geq \exp\pr{n^\delta_{\gra,j}}$ and
$\sum_{k=1}^n\sqrt{\ell^{\gra,j}_k/LL\pr{\ell^{\gra,j}_k}}\leq C_{\gra,j}\ell^{\gra,j}_n/LL\pr{\ell^{\gra,j}_n}$, 
where $\ell^{\gra,j}_n$ denotes the cardinal of $\Gamma_n^{\gra,j}$.
Define 
\begin{equation}
 \Gamma_n^{\gra,j}\pr{w}:=\ens{\gri\in \Z^d, \pr{4j+2}\gr{1}+\gra\in \Gamma_n\pr{w}}.
\end{equation}
By assumption, 
\begin{equation}
 \ell^{\gra,j}_n=\sum_{w=1}^{J_n}\operatorname{Card}\pr{ \Gamma_n^{\gra}\pr{w}}.
\end{equation}
Moreover, 
\begin{equation}
 \operatorname{Card}\pr{ \Gamma_n^{\gra,j}\pr{w}}=\prod_{q=1}^d\pr{
 \left\lfloor \frac{\overline{n}_q\pr{w,n}-a_q }{4j+2}\right\rfloor-\left\lfloor \frac{\underline{n}_q\pr{w,n}-a_q }{4j+2}\right\rfloor+1}
\end{equation}
hence by disjointness of $ \Gamma_n^{\gra}\pr{w}, 1\leq w\leq J_n$ and the assumption $\overline{n}_q\pr{w,n}-\underline{n}_q\pr{w,n}\geq 4$ we derive that 
\begin{equation}\label{eq:bound_card_gamma_nja}
 \frac{1}{\pr{4j+2}^d4^d}\ell_n\leq \ell^{\gra,j}_n\leq \frac{1}{\pr{4j+2}^d}\ell_n.
\end{equation}

In view of \eqref{eq:control_fct_max_somme_sur_union_rect} and \eqref{eq:control_fct_max_somme_sur_union_rect_2}, we derive that 
\begin{multline}\label{eq:control_fct_max_somme_sur_union_rect_3}
\norm{ 
\sup_{n\geq 1}\frac 1{ \sqrt{\ell_n LL\pr{\ell_n}  }    }
\abs{\sum_{\gri\in \Lambda_n }X_{\gr{i}}   }
}_p\\ 
\leq \sum_{j\geq 0}\frac{1}{\pr{4j+2}^{d/2}4^{d/2}} 
\sum_{\gr{0}\imd\gra \imd \pr{4j+1}\gr{1}}
 \norm{ 
\sup_{n\geq 1}\frac 1{ \sqrt{ \ell^{\gra,j}_n LL\pr{ \ell^{\gra,j}_n}  }    }
\abs{\sum_{\gri\in \Gamma_n^{\gra,j} }Y^{\gra}_{\gr{i},j}   }
}_p.
\end{multline}
We now apply Theorem~\ref{thm:linear_process_arbitrary_subsets} for each $j\geq 0$ and each $\gra\in\Z^d$ such that 
$\gr{0}\imd\gra\imd \pr{4j+1}\gr{1}$ in the independent setting, that is, $a_{\gr{0}}=1$ and $a_{\gri}=0$ for $\gri\neq\gr{0}$. 
In view of \eqref{eq:bound_card_gamma_nja}, for each $j\geq 0$, we may take $\delta_j:= \pr{4j+2}^{-d}4^{-d}$.
For $C_j$, we take $C4^d\pr{1+\log\pr{1+\pr{4j+1}^d}+LL\pr{e^2\pr{4j+1}^d}}$, as $\sup_{s>0}LL\pr{st}/LL\pr{s}\leq 1+\pr{1+\log t}+LL\pr{e^2t}$.
This completes the proof of Theorem~\ref{thm:cas_fonctionnelle_union_rectangles}.
\subsection{Proof of Corollary~\ref{cor:LLI_linear}}
 
In view of Corollary~\ref{cor:coeff_phys_dep}, it suffices to estimate $\delta_{d-1}\pr{\gri}$, that is, $\norm{f\pr{ \pr{\eps_{\gri -\gr{u}}}_{\gru\in \Z^d}   }-
f\pr{ \pr{\eps^*_{\gri -\gr{u}}}_{\gru\in \Z^d}   }        }_{2,d-1}$ where 
$f=g\circ h$ and 
$h\pr{\pr{x_{\gru}}_{\gru\in\Z^d}  }=\sum_{\gru\in\Z^d}a_{\gru}x_{\gru}$  if $\sum_{\gru\in\Z^d}a_{\gru}x_{\gru}$ converges 
(in the sense that $\lim_{m\to +\infty}\sum_{\substack{\gru\in\Z^d \\ \norm{\gru}_\infty\leq m } }a_{\gru}x_{\gru}$ exists)
and $0$ otherwise. Since $g$ is $\gamma$-Hölder continuous, the following inequality holds:
\begin{multline*}
\E{\varphi_{2,d-1}\pr{  f\pr{ \pr{\eps_{\gri -\gr{u}}}_{\gru\in \Z^d}   }-
f\pr{ \pr{\eps^*_{\gri -\gr{u}}}_{\gru\in \Z^d}   }   }   }\\ 
\leq 
\E{\varphi_{2,d-1}\pr{   \abs{ h\pr{ \pr{\eps_{\gri -\gr{u}}}_{\gru\in \Z^d}   }-
h\pr{ \pr{\eps^*_{\gri -\gr{u}}}_{\gru\in \Z^d}   }  }^\gamma     }   }
\end{multline*}
hence 
\begin{equation}\label{eq:estimate_diff_fct_de_processus_lineaire}
\E{\varphi_{2,d-1}\pr{  \frac{f\pr{ \pr{\eps_{\gri -\gr{u}}}_{\gru\in \Z^d}   }-
f\pr{ \pr{\eps^*_{\gri -\gr{u}}}_{\gru\in \Z^d}   }   }{\lambda}  }}\leq 
\E{
\varphi_{2,d-1}\pr{\frac{ \abs{a_{\gri}}^\gamma \abs{\eps_{\gr{0} }-
\eps'_{\gr{0} }}^\gamma    }{\lambda}}
}.
\end{equation}
Then, using the inequality $\varphi_{2,d-1}\pr{t^\gamma}\leq c_{d,\gamma}\varphi_{2\gamma,d-1}\pr{t}$, we derive that 
$\delta_{d-1}\pr{\gri}\leq K_{d,\gamma} \abs{a_{\gri}}^\gamma\norm{\eps_{\gr{0}}  }_{2\gamma,d-1}$, which ends the 
proof of Corollary~\ref{cor:LLI_linear}.

\subsection{Proof of Corollary~\ref{cor:fct_de_champs_Gaussiens_Hermite}}
 
Here again, we have to estimate $\delta_{d-1}\pr{\gri}$. We will essentially follows the 
steps given in Example~3 of \cite{MR3256190}, where a bound on the physical dependence 
measure with the $\mathbb L^p$-norm instead of the one given by $\varphi_{2,d-1}$ was obtained.

 First, according to Theorem~1.1.1 of \cite{MR2498953}, the expansions 
 \begin{equation}
 f\pr{Y_{\grj}}=\sum_{q=1}^{+\infty}c_q\pr{f}H_q\pr{Y_{\grj}},\quad 
  f\pr{Y^*_{\grj}}=\sum_{q=1}^{+\infty}c_q\pr{f}H_q\pr{Y^*_{\grj}}
 \end{equation}
hold in $\mathbb L^2$ hence after having extracted almost surely convergent subsequences 
$\pr{\sum_{q=1}^{Q_k}c_q\pr{f}H_q\pr{Y_{\grj}}}_{k\geq 1}$ and 
$\pr{\sum_{q=1}^{Q_k}c_q\pr{f}H_q\pr{Y^*_{\grj}}}_{k\geq 1}$, we derive by Fatou's lemma that 
\begin{multline}
\norm{ f\pr{Y_{\grj}}- f\pr{Y^*_{\grj}}}_{2,d-1}\leq 
\liminf_{k\to +\infty}\norm{ \sum_{q=1}^{Q_k}c_q\pr{f}\pr{H_q\pr{Y_{\grj}} -H_q\pr{Y^*_{\grj}}      }  }_{2,d-1}\\
\leq \sum_{q=1}^{+\infty} \abs{c_q\pr{f}}\norm{ H_q\pr{Y_{\grj}} -H_q\pr{Y^*_{\grj}}        }_{2,d-1}.
\end{multline}
Therefore, it suffices to prove that there exists a constant $C$ depending only on $d$ such that 
\begin{equation}\label{eq:goal_Hermite}
\norm{ H_q\pr{Y_{\grj}} -H_q\pr{Y^*_{\grj}}        }_{2,d-1}\leq C \sqrt{q! }q^{d-\frac 12}\norm{Y_{\grj}-Y^*_{\grj}   }_2.
\end{equation}
 First, as noticed in Example~3 of \cite{MR3256190}, the following inequality holds for all $q$: 
 \begin{equation}\label{eq:variance_Hermite}
 \norm{ H_q\pr{Y_{\grj}} -H_q\pr{Y^*_{\grj}}        }_{2}\leq \sqrt{q}\sqrt{q!}\norm{Y_{\grj}-Y^*_{\grj}   }_2,
 \end{equation}
because $Y_{\grj}$ and $Y^*_{\grj} $ have unit variance. By hypercontractivity properties (see \cite{MR1102015}, p. 65), 
the following inequality holds for all $p\geq 2$: 
\begin{equation}\label{eq:moments_p_Hermite}
 \norm{ H_q\pr{Y_{\grj}} -H_q\pr{Y^*_{\grj}}        }_{p}\leq \pr{p-1}^{q/2}\norm{ H_q\pr{Y_{\grj}} -H_q\pr{Y^*_{\grj}}        }_{2}.
\end{equation}
 Applying \eqref{eq:moments_p_Hermite} with $p=2+q^{-1}$, taking into account that $\pr{1+q^{-1}}^{q/2}\leq e^{1/2}$ 
 and using \eqref{eq:variance_Hermite} gives 
 \begin{equation}\label{eq:moments_2+1/q_Hermite}
 \norm{ H_q\pr{Y_{\grj}} -H_q\pr{Y^*_{\grj}}        }_{2+q^{-1}}\leq   \sqrt{e q \cdot q!}\norm{Y_{\grj}-Y^*_{\grj}   }_2.
 \end{equation}
Observe that for all $u\geq 0$, $\pr{1+\ln\pr{1+u}}^{d-1}\leq c_d\pr{u+1}$; applying this to $u=v^{1/q}$ and 
using the inequality $\ln\pr{1+v^{1/q}}\geq  \ln\pr{\pr{1+v}^{1/q}}=q^{-1}\ln\pr{1+v}$, we derive that 
\begin{equation}
\varphi_{2,d-1}\pr{v}\leq c_d q^{d-1}\pr{v^2+v^{2+1/q}}
\end{equation}
hence for each random variable $X$, the inequality $\norm{X}_{2,d-1}\leq c_d q^{d-1}\norm{X}_{2+q^{-1}}$ holds. Applying this 
estimate with $X=H_q\pr{Y_{\grj}} -H_q\pr{Y^*_{\grj}} $ and combining with \eqref{eq:moments_2+1/q_Hermite} gives 
\eqref{eq:goal_Hermite}, which ends the proof of Corollary~\ref{cor:fct_de_champs_Gaussiens_Hermite}.

\subsection{Proof of Corollary~\ref{cor:LLI_Volterra}}

 Let us fix $j\geq 1$ and denote $\Gca$ 
 the $\sigma$-algebra generated 
 by the random variables $\eps_{\gru}$, $\norm{\gru}_\infty\leq j$ and  define
 for  $\gr{s_1},\gr{s_2}\in\Z^d$ such that $\gr{s_1}\neq \gr{s_2}$ the 
 random variable  
 \begin{equation}
 Y_{\gr{s_1},\gr{s_2}}:=\E{\eps_{-\gr{s_1}}\eps_{-\gr{s_2} } \mid\Gca}.
 \end{equation}
Case~1: $\norm{\gr{s_1}}_\infty\geq j+1$ and $\norm{\gr{s_2}}_\infty\geq j+1$. 
The random variable $\eps_{-\gr{s_1},-\gr{s_2} }$ is independent of 
$\Gca$ hence $ Y_{\gr{s_1},\gr{s_2}}:=\E{\eps_{-\gr{s_1}} \eps_{-\gr{s_2} } }$. 
Since $\gr{s_1}\neq \gr{s_2}$ the random variables $\eps_{-\gr{s_1}}$ and 
$\eps_{-\gr{s_2}}$ are independent, we derive that $ Y_{\gr{s_1},\gr{s_2}}=0$. 

Case~2: $\norm{\gr{s_1}}_\infty\geq j+1$ and $\norm{\gr{s_2}}_\infty\leq j$. 
Since $\eps_{-\gr{s_2}}$ is $\Gca$-measurable and $\eps_{-\gr{s_1}}$ is independent 
of $\Gca$, it follows that $Y_{\gr{s_1},\gr{s_2}}=0$. 

Case~3: $\norm{\gr{s_2}}_\infty\geq j+1$ and $\norm{\gr{s_1}}_\infty\leq j$. Similarly 
as in case 2, $Y_{\gr{s_1},\gr{s_2}}=0$. 

Case~4:  $\norm{\gr{s_1}}_\infty\leq j$ and $\norm{\gr{s_2}}_\infty\leq j$. Then 
  $\eps_{-\gr{s_1}}$ and 
$\eps_{-\gr{s_2}}$ are both $\Gca$-measurable hence $ Y_{\gr{s_1},\gr{s_2}}=
\eps_{-\gr{s_1}}\eps_{-\gr{s_2}} $.

Therefore, 
\begin{equation}
\E{X_{\gr{0}}\mid \sigma\pr{\eps_{\gru}, \gru\in\Z^d, \norm{\gru}_\infty\leq j}}
= \sum_{\substack{\gr{s_1},\gr{s_2}\in\Z^d\\ 
\norm{\gr{s_1}}_\infty\leq j,\norm{\gr{s_2}}_\infty\leq j
}}a_{\gr{s_1},\gr{s_2}}\eps_{-\gr{s_1}} \eps_{-\gr{s_2} }
\end{equation}
 and consequently, 
\begin{multline}
X_{\gr{0},j}= 
\sum_{\substack{\gr{s_1} \in\Z^d\\ 
\norm{\gr{s_1}}_\infty\leq j-1 
}}\sum_{\substack{ \gr{s_2}\in\Z^d\\ 
  \norm{\gr{s_2}}_\infty=j,
}}
a_{\gr{s_1},\gr{s_2}}\eps_{-\gr{s_1}} \eps_{-\gr{s_2} }\\
+\sum_{\substack{\gr{s_1} \in\Z^d\\ 
  \norm{\gr{s_1}}_\infty = j
}}
\sum_{\substack{ \gr{s_2}\in\Z^d\\ 
  \norm{\gr{s_2}}_\infty\leq j-1}
}
a_{\gr{s_1},\gr{s_2}}\eps_{-\gr{s_1}} \eps_{-\gr{s_2} }+
\sum_{\substack{\gr{s_1} \in\Z^d\\ 
\norm{\gr{s_1}}_\infty=j 
}}\sum_{\substack{ \gr{s_2}\in\Z^d\\ 
  \norm{\gr{s_2}}_\infty=j
}}a_{\gr{s_1},\gr{s_2}}\eps_{-\gr{s_1}} \eps_{-\gr{s_2} }.
\end{multline} 
Let us control the $\el_{2,d-1}$-norm of the first term of the 
right hand side. 
Let $I$ be the set of the elements of $\Z^d$ whose $\ell^\infty$ 
norm is equal to $j$ and let $\tau\colon 
\ens{1,\dots,\operatorname{Card}\pr{I}}\to I$ be a bijection. Define the 
random variable
\begin{equation}
d_k:= \sum_{\substack{ \gr{s_1}\in\Z^d\\ 
  \norm{\gr{s_1}}_\infty\leq j-1
}}a_{\gr{s_1},\gr{\tau\pr{k}}}\eps_{-\gr{s_1}} \eps_{-\gr{\tau\pr{k} }}
\end{equation}
and the $\sigma$-algebras
\begin{equation}
\f_0:=\sigma\pr{\eps_{\gru},\gru\in\Z^d,\norm{\gru}_\infty\leq j-1   };
\end{equation}
\begin{equation}
\f_k:=\sigma\pr{\eps_{\gru},\gru\in\Z^d,\norm{\gru}_\infty\leq j-1   }
\vee \sigma\pr{\eps_{-\gr{\tau\pr{k'}}},1\leq k'\leq k  }.
\end{equation}
The sequence $\pr{d_k}_{1\leq k\leq  \operatorname{Card}\pr{I}}$ is 
a martingale differences sequence with respect to the filtration 
 $\pr{\f_k}_{0\leq k\leq  \operatorname{Card}\pr{I}}$. Moreover, using 
 independence between $\sum_{\substack{ \gr{s_1}\in\Z^d\\ 
  \norm{\gr{s_1}}_\infty\leq j-1
}}a_{\gr{s_1},\gr{\tau\pr{k}}}\eps_{-\gr{s_1}} 
 $ and $\eps_{-\gr{\tau\pr{k} }}$ and the inequality $\varphi_{2,d-1}\pr{uv}\leq 2^{d-1}\varphi_{2,d-1}\pr{u}\varphi_{2,d-1}\pr{v}$, we derive that 
$\norm{d_k}_{2,d-1} \leq 2^{d-1}  \norm{\eps_{\gr{0}}}_{2,d-1}
\norm{\sum_{\substack{ \gr{s_1}\in\Z^d\\ 
  \norm{\gr{s_1}}_\infty\leq j-1
}}a_{\gr{s_1},\gr{\tau\pr{k}}}\eps_{-\gr{s_1}}}_{2,d-1}$. 
Since
 \begin{equation}
 \E{d_k^2\mid \f_{k-1}}= \E{\eps_{\gr{0}}^2} 
 \pr{ \sum_{\substack{ \gr{s_1}\in\Z^d\\ 
  \norm{\gr{s_1}}_\infty\leq j-1
}}a_{\gr{s_1},\gr{\tau\pr{k}}}\eps_{-\gr{s_1}}}^2,
 \end{equation}
Corollary~\ref{cor:Burlholder_Orlicz} implies that 

\begin{multline}
\norm{ 
\sum_{\substack{\gr{s_1},\gr{s_2}\in\Z^d\\ 
\norm{\gr{s_1}}_\infty\leq j-1\\ \norm{\gr{s_2}}_\infty=j,
}}a_{\gr{s_1},\gr{s_2}}\eps_{-\gr{s_1}} \eps_{-\gr{s_2} }
}_{2,d-1}
\leq C_d  \norm{\eps_{\gr{0}}}_{2,d-1}\max_{\norm{\gr{s_2} }_\infty=j}
\norm{\sum_{\substack{ \gr{s_1}\in\Z^d\\ 
  \norm{\gr{s_1}}_\infty\leq j-1
}}a_{\gr{s_1},\gr{s_2}}\eps_{-\gr{s_1}}}_{2,d-1}\\
+C_d\norm{\eps_{\gr{0}}}_2\pr{ \sum_{\norm{\gr{s_2} }_\infty=j}
\norm{  \sum_{\substack{ \gr{s_1}\in\Z^d\\ 
  \norm{\gr{s_1}}_\infty\leq j-1
}}a_{\gr{s_1},\gr{\tau\pr{k}}}\eps_{-\gr{s_1}}^2}_{1,r}
}^{1/2},
\end{multline}
which entails, by Lemma~\ref{lem:norm_Orlicz_puissances}, that 
\begin{multline}
\norm{ \sum_{\substack{\gr{s_1} \in\Z^d\\ 
\norm{\gr{s_1}}_\infty\leq j-1 
}}
\sum_{\substack{ \gr{s_2}\in\Z^d\\ 
  \norm{\gr{s_2}}_\infty=j
}}a_{\gr{s_1},\gr{s_2}}\eps_{-\gr{s_1}} \eps_{-\gr{s_2} }
}_{2,d-1}\\
\leq C_d\norm{\eps_{\gr{0}}}_{2,d-1}\pr{
  \sum_{\norm{\gr{s_2} }_\infty=j}
\norm{   \sum_{\substack{ \gr{s_1}\in\Z^d\\ 
  \norm{\gr{s_1}}_\infty\leq j-1
}}a_{\gr{s_1},\gr{s_2}}\eps_{-\gr{s_1}} }^2_{2,r}
}^{1/2}.
\end{multline}
We conclude by applying an other time Corollary~\ref{cor:Burlholder_Orlicz}.

\subsection{Proof of Corollaries~\ref{cor:LLI_linear_sum_on_disjoint_unions_of_rectangles}, \ref{cor:fct_de_champs_Gaussiens_Hermite_sum_on_rect} and \ref{cor:LLI_Volterra_sum_on_rect}}

These corollaries follow from Corollary~\ref{cor:coeff_phys_dep_union_of_rectangles} and an estimation of the dependence 
coefficients in the same spirit as in the proof of Corollaries~\ref{cor:LLI_linear}, \ref{cor:fct_de_champs_Gaussiens_Hermite} 
and \ref{cor:LLI_Volterra_sum_on_rect}. For the sake of completeness, we give only the key steps of the proof that need 
to be modified. 

For Corollary~\ref{cor:LLI_linear_sum_on_disjoint_unions_of_rectangles}, we only need to use a 
version of \eqref{eq:estimate_diff_fct_de_processus_lineaire} with $\varphi_{2,d-1}$ replaced by 
$\varphi_{2,0}$. 

For Corollary~\ref{cor:fct_de_champs_Gaussiens_Hermite_sum_on_rect}, we can use directly \eqref{eq:variance_Hermite}.

For Corollary~\ref{cor:LLI_Volterra_sum_on_rect}, we can simply use orthogonality of the increments of a martingale difference 
sequence instead of the inequality given in Corollary~\ref{cor:Burlholder_Orlicz}.

\bigbreak 

\textbf{Acknowledgement} This research work is supported by the DFG Collaborative Research Center SFB 823 `Statistical modelling of nonlinear dynamic processes'.

The author would like to thank the referee for remarks improving the quality of the manuscript and 
for the suggestion of considering the summation on disjoin unions of rectangles.


\def\polhk\#1{\setbox0=\hbox{\#1}{{\o}oalign{\hidewidth
  \lower1.5ex\hbox{`}\hidewidth\crcr\unhbox0}}}\def\cprime{$'$}
  \def\polhk#1{\setbox0=\hbox{#1}{\ooalign{\hidewidth
  \lower1.5ex\hbox{`}\hidewidth\crcr\unhbox0}}} \def\cprime{$'$}
\providecommand{\bysame}{\leavevmode\hbox to3em{\hrulefill}\thinspace}
\providecommand{\MR}{\relax\ifhmode\unskip\space\fi MR }
\providecommand{\MRhref}[2]{%
  \href{http://www.ams.org/mathscinet-getitem?mr=#1}{#2}
}
\providecommand{\href}[2]{#2}

\end{document}